\documentclass[10pt]{article}
\usepackage{amsmath,amsthm,amscd,amssymb,mathrsfs,setspace}
\usepackage{latexsym,epsf,epsfig}
\usepackage[hmargin=3cm,vmargin=3.5cm]{geometry}
\usepackage{color}
\usepackage{bm}
\usepackage{hyperref}

\newcommand{\ds}{\displaystyle}

\newcommand{\R}{{\mathbb R}}
\newcommand{\reals}{\mathbb{R}}
\newcommand{\realstwo}{\mathbb{R}^2}
\newcommand{\realsthree}{\mathbb{R}^3}
\newcommand{\xb}{{\bf{x}}}

\newcommand{\Dn}{\partial_{\nu}}
\newcommand{\Dz}{\partial_z}

\newcommand{\cF}{{\mathcal{F}}}

\newcommand{\bH}{\mathbf{H}}
\newcommand{\Om}{\Omega}

\newcommand{\s}{\sigma}
\newcommand{\e}{\epsilon}
\newcommand{\Ez}{E_{z}}
\newcommand{\bA}{\mathbf{A}}
\newcommand{\cD}{\mathscr{D}}
\newcommand{\lb}{\langle}
\newcommand{\rb}{\rangle}
\newcommand{\bT}{\mathbb T}
\theoremstyle{plain}
\newtheorem{condition}{Condition}
\newtheorem{conjecture}{Conjecture}
\newtheorem{theorem}{Theorem}[section]
\newtheorem{lemma}[theorem]{Lemma}

\newtheorem{corollary}[theorem]{Corollary}
\newtheorem{remark}[theorem]{Remark}

\theoremstyle{definition}
\newtheorem{definition}[theorem]{Definition}
\numberwithin{equation}{section}
\numberwithin{theorem}{section}
\title{Attractors and Determining Functionals for A Flutter Model: \\ Finite Dimensionality Out of Thin Air }

\date{\today}
\normalsize
 \author{ 
            Justin T. Webster \\
 University of Maryland, Baltimore County \\
 \it Baltimore, MD \\
  \it websterj@umbc.edu}

\begin{document}

\maketitle

\quad\\[-0.5cm]

\begin{abstract}
We establish the effective {\em finite dimensionality} of the dynamics corresponding to a flow-plate interaction PDE model arising in aeroelasticity: a  nonlinear panel, in the absence of rotational inertia, immersed in an inviscid potential flow.  An intrinsic component of the analysis  is  the study of a plate equation with a {\it delay} term---a fundamentally non-gradient dynamics. 
First, we construct a compact global attractor and observe that the attractor is smooth, with finite fractal dimension in the state space. Secondly, by fattening the attractor, we obtain an exponential attractor, though with finite dimension only in an extended space. Lastly, we show that a finite set of {\em determining functionals} exists by considering the {\em completeness defect} for some practical functionals on $H_0^2(\Omega)$ (e.g., nodes, modes, and averages). The primary tool here is the recent quasi-stability theory of Chueshov and Lasiecka. All of the main results require {\em no imposed structural damping}, as dissipative effects are contributed by the flow through the coupling. In the final section, we discuss additional results and conjectures when imposed structural damping is present. 
\end{abstract}\quad\\


\section{Introduction}
In this paper we consider a canonical fluid-structure interaction: the flow of gas over a lower dimensional surface. Specifically, we analyze the coupled partial differential equation (PDE) model of an inviscid potential flow over a clamped plate; we refer to the physical configuration as a {\em panel}. As we describe below, it is well-known that the presence of the 3-D gas flow can be destabilizing for the 2-D elastic plate embedded in the flow domain's boundary. Indeed, {\em aeroelastic flutter} is a particular type of flow-induced destabilization resulting from systemic bifurcation, and it often yields sustained {\em limit cycle oscillations} (LCOs) in the structure. From a mathematical point of view, the model represents a hyperbolic-like flow equation strongly coupled to a nonlinear plate equation (coming from the theory of large deflections \cite{springer,ciarlet,lagnese}). The model presented below is established in the engineering literature \cite{dowell,dowellnon,dowell1}, and has also been well-studied in the applied PDE literature, predominantly by Chueshov et al. \cite{LBC96,b-c-1,oldchueshov1,chuey,springer,ryz,ryz2,oldchueshov2}. These aforementioned studies (all before 2010) have typically involved the addition of regularizing and/or imposed damping mechanisms; as we describe below, this is a quite natural first step in the analysis. After 2010, with the involvement of the present author and the development of Chueshov and Lasiecka's theory of {\em quasi-stability} (Section \ref{quasisec}), a variety of new results and approaches appeared \cite{jadea12,delay,supersonic,webster,conequil1,conequil2}. The surveys \cite{fereisel} and \cite{survey2} provide a mathematical discussion of previous results, while \cite{survey1} provides a discussion of the connections between engineering and mathematical analyses of this and related models.

The principal purpose of this paper is to rigorously explore notions of asymptotic {\em finite dimensionality} for the flow-plate model in the absence of mathematically helpful terms (i.e., as it is given in classical engineering references \cite{bolotin,dowellnon}). That the model below is well-posed is established \cite{webster,jadea12,supersonic}---see \cite{survey2}. Here, we ask after the qualitative properties of the dynamics in the non-transient regime---including the possibility of unstable, post-flutter type behaviors. This includes, for instance, the possibility of chaotic dynamics or convergence to equilibrium, in addition to LCO behaviors. We show in the work at hand that the non-transient regime is truly finite dimensional. 

It is well-established in the engineering literature that flutter is a {\em low dimensional} phenomenon \cite{dowell,dowell1,vedeneev,vedeneev2}. This is to say that engineers only utilize small numbers of ``modes" to describe the asymptotic-in-time behaviors of the flow-plate system, and justify this empirically, in an a posteriori fashion. In line with the above discussion, we rigorously examine this claim for a specific flutter model. This is to say that we begin from the fully infinite dimensional PDE flow-plate system {\em as it appears in the engineering literature}, and we rigorously demonstrate (in a variety of ways) that the essential, long-time dynamics are {\em finite dimensional}. In particular, this will be done with attracting sets and sets of determining functionals, described precisely below in Section \ref{tools}. 

We also take the opportunity to point out a classic pair of mathematically-oriented papers, motivating much of what is here: \cite{Holmes,HolMar78}. These papers study a 1-D, simplified version of the structural model given here in \eqref{reducedplate} (with $q\equiv0$); the former, \cite{Holmes}, makes an a priori truncation of the PDE system and studies the dynamical system properties of the low dimensional systems (as is common in engineering \cite{vedeneev,vedeneev2,dowellchaos}, for instance). The sequel, \cite{HolMar78}, studies attractors and inertial sets (with the available technology of the time), giving a rigorous justification that the model can be studied from a finite dimensional (albeit with $N$ large) point of view---a sort of vindication of the earlier work.

\subsection{Goals of the Paper}

With the above established, we assert that the {\bf main goals} of this paper are: 
{\bf (i)} To demonstrate quantitatively what engineers often state qualitatively---that flutter is a finite dimensional phenomenon. 
{\bf (ii)} To establish a robust set of results concerning the asymptotic-in-time behavior of the flow-plate system in the absence of mathematically helpful terms (imposed damping or regularizations) ; some of these are novel, and some that are not have novel proofs here.
{\bf (iii)} To showcase the power and ease of applicability of the recent quasi-stability theory of Chueshov and Lasiecka \cite{springer,quasi} as the main tool for most of our principal results here.

Below, we provide a complete exposition of the steps between the full flow-structure model to various notions of finite dimensional end behavior. This ``reduction" is accomplished {\em without any imposed damping} whatsoever in the model. We will prove the main points, and, when we omit proofs, we provide explicit references and discussion of the underlying theory. Also, some results/proofs herein are not novel, but are included for self-containedness and to show the reader {\em precisely how finite dimensionality cuts in} in as many ways as possible. 

\subsection{Flow-plate Interactions in Application}
	The interactive dynamics between a fluid flow and a solid embedded in a lower dimensional interface has been a topic of immense activity for 50 years \cite{dowell1,survey1,survey2,fereisel} (and many references therein). Theoretical, numerical, and experimental scientists are interested in characterizing, predicting, and controlling flow-structure behavior. Here, we consider model that specifically captures aeroelastic {\bf flutter} \cite{bolotin,dowell1,dowellnon,vedeneev,vedeneev2}:  flutter is a self-excitation instability that occurs through a feedback between displacements of an elastic structure and dynamic pressure changes of a surrounding fluid flow. Certain flow velocities bring about a bifurcation in the structural dynamics \cite{Holmes,HolMar78}---stable dynamics may become oscillatory, in the form of {\em limit cycle oscillations} (LCOs), or even chaotic  \cite{dowellchaos}. Flutter can occur in a multitude of applications, including: buildings and bridges in wind, aircraft structures and paneling, pipes conveying fluid, and even in the human respiratory system---see \cite{survey1}. 
	
With respect to flight, flutter instabilities are of paramount concern in the supersonic and transonic flight regimes; from a design point of view, flutter cannot be overlooked due to its potential effects on the structure due to fatigue and/or large amplitude response. The standard panel flutter system shown in the next section has been utilized in a large body of work, and was originally introduced to describe projectile paneling \cite{bolotin,dowellnon}.  A majority of corresponding scholarly work has been {\em computational} in nature \cite{dowell1,dowell,dowellnon}. Indeed, given the difficulty of modeling and analyzing coupled PDEs at an interface \cite{LT}, theoretical results have been comparatively few. While numerical studies are incredibly important, and provide vital qualitative information, they are based on finite dimensional approximations of continuum models fundamentally described by PDEs. Ad hoc, a priori truncations of infinite dimensional models should be justified in some rigorous sense.

Being dictated by physics, flow-plate interaction models do not typically yield straightforward functional setups. Serious PDE problems include: (i) the mismatch of regularity between dynamics and/or hyperbolic-hyperbolic coupling, (ii) the appearance of ill-defined boundary traces, and/or (iii) time-evolving domains. And, while linear theory is viable to predict the onset of instability \cite{vedeneev}, capturing post-flutter dynamics requires structural nonlinearity \cite{dowellnon,survey2}, and constitutes a challenging analytical task.   With respect to the latter point, consistent with engineering literature \cite{dowellnon,HolMar78}, we employ the theory of large deflections \cite{ciarlet,lagnese}. Beginning with von Karman theory, we invoke the Berger simplification \cite{berger,gw,studyberger,studyberger2}, widely accepted, and often used, for the panel configuration \cite{studyberger2}.

	\subsection{Mathematical Model and Energies}
	
	The classic {\em flutter} model \cite{dowellnon} takes a inviscid, irrotational flow of compressible gas in $\realsthree_+=\{\mathbf x=(x,y,z)~:~ z > 0\}$, with an elastic panel embedded in the flow boundary $\partial \realsthree_+$. The unperturbed flow velocity has magnitude $U \in \mathbb R$ in the  $x$-direction; we have scaled $U=1$ to Mach 1, so $0 \le U <1$ corresponds to subsonic flow. The equilibrium position of the plate is modeled by a bounded domain $\Omega \subset \{\xb: z = 0\}$, with smooth boundary $\partial \Omega = \Gamma$ and associated unit outward normal $\nu=\nu(x,y)$.
 
The scalar function $u: \Omega \times \R_+ \to \reals$ represents the transverse, Lagrangian displacement of the plate in the $z$-direction at $(x,y)$ at the moment $t$. The flow is of potential type, with $\phi:\realsthree_+\times \reals_+ \rightarrow \reals$ the perturbation velocity potential
 \cite{bolotin,dowell1}, so the flow field $\mathbf v$ on $\mathbb R^3_+$ is given by $\mathbf v = U\mathbf e_1+\nabla \phi$.  The strong coupling occurs (i) in the dynamic pressure term $p(\mathbf x,t)=p_0(\mathbf x)+[\phi_t+U\phi_x]\big|_{\Omega}$, which contains the static pressure and the acceleration potential of the flow, and (ii) in the Neumann condition (the {\em downwash}) of the flow; the latter includes the material derivative of the structure which accounts for the Eulerian-to-Lagrangian change of variables \cite{bolotin,dowell1}.
	\begin{equation}\label{flowplate}\begin{cases}
u_{tt}+\Delta^2u+ku_t+f(u)= p_0(\mathbf x)+r_{\Omega}\big[tr(\phi_t+U\phi_x)\big]& \text { in } \Omega\times (0,T),\\
u(t=0)=u_0;~~u_t(t=0)=u_1,\\
u=\Dn u = 0 & \text{ on } \Gamma\times (0,T),\\
(\partial_t+U\partial_x)^2\phi=\Delta \phi & \text { in } \realsthree_+ \times (0,T),\\
\phi(t=0)=\phi_0;~~\phi_t(t=0)=\phi_1,\\
\partial_z \phi = \big[(\partial_t+U\partial_x)u \big]_{\text{ext}} & \text{ on } \realstwo_{(x,y)} \times (0,T).
\end{cases}
\end{equation}
The notation $tr(\cdot)$ corresponds to the trace operator $H^1(\mathbb R^3_+) \mapsto L^2(\mathbb R^2)$, while $r_{\Omega}:L^2(\mathbb R^2) \to L^2(\Omega)$ corresponds to the restriction to $\Omega$, with corresponding extension by zero. For functions in $H_0^2(\Omega)$ we denote that extension by $u_{\text{ext}} \in H^2(\mathbb R^2)$ (this action is regularity preserving in this configruation). 
\begin{remark}
It is immediately obvious from \eqref{flowplate} that, if~$\phi_t \in L^2(\mathbb R_+^3)$ only, the dynamic pressure $p(\mathbf x,t) = p_0+r_{\Omega}tr(\phi_t+U\phi_x)$ cannot be interpreted through the standard trace theorem here; as we will see, hidden regularity to interpret this trace will be necessary.
\end{remark}

 The  nonlinearity of principal interest here is that of Berger, of extensible, cubic type \cite{lagnese,ciarlet,dowellnon}\footnote{In the case of beams, this type of nonlinearity is often referred to as Krieger-Woinowsky or even Kirchhoff---see \cite{HTW,HHWW} and \cite{menz2,menz3} for more discussion}: 
\begin{equation}\label{berger}
f(u)=f_B(u) = [b_1 - b_2||\nabla u||^2]\Delta u.
\end{equation}
The parameter $b_1\in \mathbb R $ is a pre-stressing parameter \cite{dowellchaos,studyberger}, corresponding to equilibrium in-plane forces, while $b_2 > 0$ scales the strength of the nonlinear restoring force, the term itself depending on {\em local stretching}; when $b_2=0$ the model is linear. The parameter $k\ge 0$ corresponds to weak (or viscous) structural damping, and {\em for most of the paper will be taken to be zero}. 

Denoting standard $L^2$ norms on a domain $\mathcal O$ by $||\cdot||_{L^2(\mathcal O)}$, and using $(\cdot,\cdot)_{\mathbb R^3}$ and $\langle \cdot , \cdot \rangle_{\Omega}$ as the inner product notations, the  {\it plate energy} is defined as usual \cite{springer,lagnese}:
\begin{align}\label{plateenergy}
E_{pl}(u) =& \dfrac{1}{2}\big[\|u_t\|_{L^2(\Omega)}^2+ \|\Delta u\|_{L^2(\Omega)}^2 \big] +\Pi(u).
\end{align}
\noindent $\Pi(u)$ is a potential of the nonlinear and nonconservative forces,  given
by
 \begin{equation}\label{Pi}
 \Pi(u)=\Pi_B(u)=\dfrac{b_2}{4}||\nabla u||^4_{L^2(\Omega)}-\dfrac{b_1}{2} ||\nabla u||^2_{L^2(\Omega)}-\langle p_0,u\rangle_{\Omega}.
 \end{equation}
 The natural energies associated with subsonic flow and interactive dynamics are given below:
\begin{align}
 E_{fl}(\phi) =&  ~\dfrac{1}{2}\big[\|\phi_t\|_{L^2(\R^3_+)}^2-U^2\|\partial_x\phi\|_{L^2(\R^3_+)}^2+\|\nabla \phi\|_{L^2(\R^3_+)}^2\big],\label{flowenergy} \\
 E_{int}(u,\phi)=&~2U\lb tr[\phi],u_x\rb_{\Omega},~~0 \le U<1.
\end{align}
 The total (unsigned) energy is then defined to be \begin{equation}\label{totalenergy}
\mathcal E(u(t),\phi(t))=\mathcal E(t) = E_{pl}(u(t))+E_{fl}(\phi(t))+E_{int}(u(t),\phi(t)).
\end{equation}
\noindent {\bf Notation}: From this point onward, we often denote $L^2$ norms with no subscript, with the meaning clear from the context; Sobolev norms on the standard space $H^s(\mathcal O)$ will be written
$||f||_{H^s(\mathcal O)}=||f||_s$, and we identify $||f||_{L^2(\mathcal O)}=||f||_0=||f||,$ with $\mathcal O =\Omega$ or $\mathbb R_+^3$ and the meaning clear from the context.

\begin{remark} It is clear that with $U>1$ the above energy degenerates; we then define a supersonic energy $$E_{fl}^{sup}(\phi(t)) \equiv \dfrac{1}{2}\Big[||\nabla \phi(t)||^2 + ||\phi_t+U\phi_x||^2\Big],~~\text{ with }~E_{int}^{sup} \equiv 0.$$  This modified topological measure of solutions is the correct one for supersonic well-posedness \cite{supersonic}, though this is not critical to our discussions here.
\end{remark}
 We will also need to consider {\em positive energies}, so we define
\begin{align}\label{posen}
\Pi_*(u)=&\dfrac{b_2}{4}||\nabla u||^4,\\
E_*(u) =&  \dfrac{1}{2}[||u_t||^2+||\Delta u||^2]+\Pi_*(u),\\ \mathcal E_*(u,\phi) =& E_*(u)+E_{fl}(\phi).
\end{align}
\noindent According to these norms, the natural energy space for the dynamics $(u,u_t;\phi,\phi_t)$ is then\footnote{$W^1(\mathbb R^3_+)$ is a homogeneous Sobolev space given as the subspace of $L^2_{loc}(\mathbb R_+^3)$ with finite gradient norm as in \eqref{flownorm}.}: 
\begin{equation}\label{energyspace1}
Y = Y_{pl} \times Y_{fl}\equiv \left(H_0^2(\Omega)\times L^2(\Omega)\right)\times\left(W^1(\realsthree_+) \times L^2(\realsthree_+)\right),
\end{equation}
 defined through the norm
\begin{equation}\label{flownorm}
||(u,v;\phi,\psi)||^2_Y = ||\Delta u||^2+||v||^2+||\nabla \phi||^2+||\psi||^2.
\end{equation}
We will also consider a stronger space below:
\begin{equation}\label{energyspace2}
Y_s\equiv \left(H_0^2(\Omega)\times L^2(\Omega)\right)\times\left(H^1(\realsthree_+) \times L^2(\realsthree_+)\right).
\end{equation}


\subsection{Outline of the Remainder of the Paper}
Section \ref{wellp}  provides a discussion of well-posedness and basic notions about the solution semigroup and associated bounds. Section \ref{techresults} gives the main results in this paper, with narrative structure; this section also precisely discusses the previous mathematical work on the model \eqref{flowplate} (and the associated reduced model \eqref{reducedplate}). Section \ref{tools} gives the technical tools needed used in proving the main theorems, including an overview of quasi-stability theory. Section \ref{reducedsection} gives the rigorous reduction of the flow-plate system in \eqref{flowplate} to a delayed plate equation \eqref{reducedplate}. Section \ref{techstuff} establishes the main estimates (observability, absorbing ball, and quasi-stability) supporting the main theorems' proofs. After this supporting work,  Section \ref{proofs} provides the proofs of each of the main theorems, in short subsections. The final section, Section \ref{finale}, describes  conjectures and open problems when structural damping is imposed.

\section{Well-posedness and Fundamental Notions}\label{wellp}

The above flow-plate dynamics, cast in the appropriate framework, are well-posed \cite{LBC96,springer,jadea12,supersonic,webster}. For precise definitions of {\em strong, generalized (semigroup), and weak} solutions consult \cite{jadea12,supersonic,webster}.  The following  result is established in \cite{webster,jadea12} for $0\le U<1$ and in \cite{supersonic} for $U>1$:
\begin{theorem}[Nonlinear Semigroup]
\label{nonlinearsolution} Assume $U \neq 1$, $p_0 \in L^2(\Omega)$. Take $b_1 \in \mathbb R$ with $k \ge 0$ and $b_2 > 0$.  Then for any $~T>0$,  \eqref{flowplate} has a unique strong (resp. generalized, weak) solution on $[0,T]$, denoted by $S_t (y_0)$, for $y_0=(u_0,u_1;\phi_0,\phi_1) \in Y$. In the case of strong solutions, the natural compatibility condition must be in force: $ \partial_z \phi_0 = [u_1+U\partial_xu_{0}]_{\text{ext}}$.
Moreover,  $(S_t, Y)$ and $(S_t,Y_s)$ are dynamical systems.
\end{theorem}
In the subsonic case, {\bf $0 \le U<1$}, more can be said.
\begin{theorem}[Subsonic Flows]\label{subsonictheorem}
In addition to the hypotheses of Theorem \ref{nonlinearsolution}, take $U \in [0,1)$. Then all solutions satisfy the following energy  {\it equality}:
 \begin{equation}\label{eident}\ds {\mathcal{E}}(t)+k\int_s^t ||u_t(\tau)||_{L^2(\Omega)}^2 d\tau= {\mathcal{E}}(s)\end{equation} for $t>s$.
Moreover, there exists a constant $C(||y_0||_{Y})$ such that for all $ t \geq 0 $ we have: 
\begin{equation}\label{stableS}
 \|S_t (y_0)\|_Y \leq C \left(\|y_0\|_{Y}\right).
 \end{equation}
In addition, the  semigroup $S_t$ is locally Lipschitz on $Y$
\begin{equation}\label{lip}
||S_t(y_1) - S_t(y_2) ||_Y \leq C (R,T) ||y_1-y_2||_Y,~~ \forall~ ||y_i||_Y \leq R,~~ t \leq T 
\end{equation}
\end{theorem}

For the above semigroup we introduce the dynamics operator $\bT:\mathscr{D}(\bT)\subset Y_s \to Y_s$. For its precise structure, we give reference to \cite{jadea12,supersonic,webster}. We do have: \begin{equation}\label{domainprop}\cD(\bT) \subset  (H^4\cap H_0^2)(\Omega)\times H_0^2(\Omega) \times H^2(\realsthree_+)\times H^1(\realsthree_+). \end{equation}
\begin{remark}
The natural {\em invariance} of the dynamics is with respect to the norm $||\cdot||_{Y}$. However, via 
\begin{equation}\label{oneusing*}||\phi(t)||_{L^2(\realsthree_+)} \le ||\phi_0||_{L^2(\realsthree_+)}+\int_0^t||\phi_t(\tau)||_{L^2(\realsthree_+)} d\tau,\end{equation} invariance in $Y_s$ can be recovered {\em on finite time intervals}.
\end{remark}

In order to describe the dynamics of the flow in the context of long-time behavior it is necessary to introduce local spaces, denoted by $Y_{fl,\rho}$: $$\|(\phi_0,\phi_1)\|_{Y_{fl},\rho}\equiv \int_{ K_{\rho} } |\nabla \phi_0|^2  + |\phi_1|^2 d\xb,$$  
where $K_{\rho} \equiv \{ \xb \in \mathbb R^3_{+}; |\xb |\leq \rho \} $. We denote by $Y_{\rho}\subset Y$ the space $Y_{pl}\times Y_{fl,\rho}$. By virtue of the Hardy inequality \cite[p.301]{springer}:
$$\|(\phi_0,\phi_1)\|_{Y_{fl},\rho}^2 \le  \|(\phi_0,\phi_1)\|_{H^1(K_{\rho})\times L^2(K_{\rho})}^2\le \|(\phi_0,\phi_1)\|_{Y_{fl}}^2.$$

We now highlight the boundedness (from below) of $\mathcal E$.  Such a bound is necessary to obtain the semigroup stability in Theorem \ref{nonlinearsolution}. 
First, we have \cite[Lemma 5.2, p. 3136]{webster}:
\begin{lemma}\label{energybound} Let the hypotheses of Theorem \ref{subsonictheorem} be in force. 
Then for generalized solutions to \eqref{flowplate}, there exist positive constants $c,C,$ and $M$ positive such that
\begin{equation}  c \mathcal E_*(t)-M_{p_0,b_1,b_2} \le \mathcal E(t) \le C \mathcal E_*(t)+M_{p_0,b_1,b_2}  \end{equation}
\end{lemma}
The proof of Lemma \ref{energybound} given in \cite{webster}  relies on two estimates controlling lower frequencies.  The first estimate controls interactive energy $E_{int}$, on the strength  of   Hardy inequality.
\begin{lemma}\label{EIN} For $\phi \in W^1(\mathbb R^3_+)$ and $u \in H^1(\Omega)$.
 \begin{equation} |E_{int}(t)| \le \delta \|\nabla \phi(t)\|_{L^2(\realsthree_+)}^2+C{U^2}{\delta^{-1}}\|u_x(t)\|_{L^2(\Omega)}^2, ~~\delta>0, \end{equation}
 \end{lemma}
The next critical estimate  controls low frequencies  \cite[p. 49]{springer} by exploiting superlinearity:
\begin{lemma}\label{l:epsilon}
For any $u \in H^2(\Omega) \cap H_0^1(\Omega) $ and   $\eta,\epsilon > 0 $ there exists $M_{\epsilon,\eta} $ such that
$$\|u\|^2_{H^{2-\eta}(\Omega)} \leq \epsilon [\|\Delta u\|^2_{L^2(\Omega)}  + \Pi_*(u) ] + M_{\eta,\epsilon}$$
\end{lemma}
From the above lemmata  and energy inequality  we  have \cite{jadea12,webster}:
\begin{corollary}\label{globalbound} Take the hypotheses of Theorem \ref{subsonictheorem}. Then any generalized solution to \eqref{flowplate} satisfies \begin{equation}\label{apr}
\sup_{t \ge 0} \left\{\|u_t\|_{L^2(\Omega)}^2+\|\Delta u\|_{L^2(\Omega)}^2+\|\phi_t\|_{L^2(\realsthree_+)}^2+\|\nabla \phi\|_{L^2(\realsthree_+)}^2 \right\}  \leq C\big(\|y_0\|_Y\big)< + \infty.\end{equation}
In addition, if $k>0$,  then the dissipation integral is finite: we have \begin{equation}\label{dissint} \int_0^{\infty} \|u_t(t)\|_{L^2(\Omega)}^2 dt \leq  K(y_0) < \infty.\end{equation}

\end{corollary}

\section{Main Results and Discussion}\label{techresults}
There are four main results in this paper, which we first describe informally.

The first main result shown here is that we can reduce the dynamical system $(S_t,Y)$, associated to solutions of \eqref{flowplate}, to a delayed plate system $(T_t,\bH)$. The reduction also brings with it some natural damping from the flow-plate coupling. The proof of this result was shown earlier with rotational inertia in the plate dynamics \cite{b-c-1,springer}, and given in \cite{delay} for the model here, albeit with minimal details.
 The remaining three results concern the long-time behavior of the {\em plate component} of the flow-plate system. Namely, without imposing any structural damping, we have three notions of finite dimensional end behavior for the plate component $(S_t,Y)$. 

First, we show that a compact global attractor exists---this is a compact set in the phase space that also happens to be smooth and finite dimensional; it is fully invariant and uniformly attracts all bounded sets. This result was first shown in \cite{delay}, but the proof at hand is streamlined by taking advantage of the structure of Berger's nonlinearity. Secondly, we show that by ``fattening" the attractor, we obtain a forward invariant set in the phase space that attracts all bounded sets with {\em exponential rate}, though the finite dimensionality of this set may be in a weaker topology than $Y_{pl}$. Both of these results produce ``nice" sets which somehow fully capture the essential non-transient behavior of the flutter system in \eqref{flowplate}, while also being fundamentally finite dimensional. {\em This is to say that LCOs associated to flutter, as non-stationary end behaviors, are contained in the attractor.} These results are implied by abstract statements in \cite{springer}, but are not written out explicitly, as we do here. Moreover, the proofs here rely critically on obtaining the quasi-stability estimate on a bounded, forward invariant set---not the approach given in \cite{springer} (and references therein).

Finally, the last result concerns {\em determining functionals}. The existence of such functionals gives a practical means of uniquely characterizing time-asymptotic behavior of solutions. Indeed, as the structure of the attractor can be quite complex, and finite dimensionality estimates are often inflated, making direct use of the attractor's finite dimensionality is difficult. 
We show that, through the stabilizability estimate, finite sets of nodal values, modal coefficients, or local volume averages, provide determining functionals. That is, {\em these practical, finite collections uniquely determine trajectories, providing a sufficient set of statistics for characterizing global end behaviors}. This is a {\bf new result for this system} and is not directly implied by previous work, and we explicitly provide the construction.

All three results on finite dimensionality, presented here, have proofs which critically rely on the notion of a quasi-stable dynamical system. In fact, one might say that the results and proofs here provide a clear advertisement for the clarity and power of quasi-stability theory---specifically, when one can obtain the quasi-stability estimate on an absorbing ball.

	\subsection{Definition of Main Objects} Let $H$ be a Hilbert space with $(S_t,H)$ an associated dynamical system. 
	
		The {\it fractal} (box-counting) {\it dimension} of a set $A \subset H$, denoted ${\rm dim}_f A$, is defined by
\[
{\rm dim}_f A=\limsup_{\e\to 0}\frac{\ln n(A,\e)}{\ln (1/\e)}\;,
\]
where $n(M,\e)$ is the minimal number of closed balls in $H$ of the
radius $\e$ covering the set $M$. By Ma\~{n}\'{e}'s Theorem, a set that has finite fractal dimension can be embedded into some $\mathbb R^n$, and thus can be injected as a subset of some higher dimensional Euclidean space \cite{CRV}.

	We recall that (see, e.g., \cite{Babin-Vishik,ch-0,lad}) for the system $(S_t,H)$,
a compact \textit{global attractor} $\mathbf{A} \subset \subset H$ is an invariant set (i.e., $S_t\mathbf{A}=\mathbf{A}$ for all
$t \ge 0 $) that uniformly attracts bounded sets $B\subset H$:
\begin{equation}  \label{dist-u}
\lim_{t\to+\infty}d_{H}\{S_t B|\bA\}=0,~~~\mbox{where}~~
d_{H}\{S_t B|\bA\}\equiv\sup_{y\in B}{\rm dist}_{H}(y,\bA),
\end{equation}
As we will see,  often the compact attractor $\bA$ is more regular than $H$, with $\text{dim}_f\mathbf A < \infty$.

A generalized fractal {\em exponential attractor} for the dynamics $(S_t,H)$ is a {\em forward invariant}, compact set, $ A_{\text{exp}} \subset H$ with finite fractal dimension that attracts bounded sets (as above) {\em with uniform exponential rate} in $H$.   The word ``generalized" is included to indicate that the finite dimensionality is perhaps in a weaker topology ($||\cdot ||_{\widetilde H}$) than that of state space ($||\cdot||_H$).


 Lastly, let $\mathscr L = \{l_j~:~j=1,...,N\}$ be a finite set of continuous, linear functionals on $H$ (or some component of $H$, if it is a product space). We say that $\mathscr L$ is a (an asymptotically) determining set of functionals for $(S_t,H)$ if the following condition holds:
 $$\Big(\lim_{t \to \infty} | l_j(S_ty_1)-l_j(S_ty_2)| = 0\Big)~~(\forall~~j=1,...,N)~~\implies~~\lim_{t \to \infty} ||S_ty_1-S_ty_2||_H =0.$$

	\subsection{Delayed Dynamical System}\label{reducedresults}
	 Below, we make use of the notation $u^t=\{u(t+s)~:~~s \in (-t^*,0)\}$ for some fixed {\em time of delay}, $t^*>0$. 
\begin{theorem}[Delayed Dynamical System]\label{rewrite}
Let the hypotheses of Theorem~\ref{nonlinearsolution} be in force,
and $(u_0,u_1;\phi_0,\phi_1) \in H_0^2(\Omega) \times L^2(\Omega) \times H^1(\realsthree_+) \times L^2(\realsthree_+)$.
Assume there exists a $\rho_0>0$ such that ~$supp(\phi_0),supp(\phi_1) \subset K_{\rho_0}$.
Then the there exists a time $t^{\#}(\rho_0,U,\Omega) > 0$ such that for all $t>t^{\#}$ any weak solution $u(t)$ in
(\ref{flowplate})  satisfies the following equation (in a weak sense):
\begin{equation}\label{reducedplate}
u_{tt}+\Delta^2u+ku_t+f(u)=p_0-(\partial_t+U\partial_x)u-q(u^t),
\end{equation}
with
\begin{equation}\label{potential}
q(u^t)=\dfrac{1}{2\pi}\int_0^{t^*}\int_0^{2\pi} M^2_{\theta} [u_{\text{ext}}(\xb(U,\theta,s), t-s)]d\theta ds.
\end{equation}
Here, ~$M_{\theta} \equiv \sin(\theta)\partial_x+\cos(\theta) \partial_y$, ~$\xb(U,\theta,s) = \big(x-(U+\sin \theta)s,y-s\cos\theta\big) \subset \realstwo$, and 
 \begin{equation}\label{delay} t^*\equiv \inf \{ t~:~\xb(U,\theta, s) \notin \Omega \text{ for all } \xb \in \Omega, ~\theta \in [0,2\pi], \text{ and } s>t\}.
\end{equation} 
\end{theorem}
The structure of the delay potential $q(\cdot)$ comes from the explicit solver for the potential flow equation (with Neumann data $[u_t+Uu_x]_{ext}$) on $\mathbb R^3_+$. 
\begin{remark} The system given in \eqref{reducedplate}--\eqref{potential} (taken with appropriate initial conditions) is independently well-posed \cite{delay}. This is to say that the \eqref{genplatedelay} below is well-posed in the appropriate delay sense, and generates a delay dynamical system. (This is discussed at length in Section \ref{delayplatesec}.)
\begin{equation}\label{genplatedelay}\begin{cases}
u_{tt}+\Delta^2u+k_0u_t+f_B(u)= p_0+Lu+q(u^t,t) ~~ \text { in } ~\Omega\times (0,T), \\
u=\Dn u = 0  ~~\text{ on } ~ \Gamma \times (0,T),  \\
u(0)=u_0,~~u_t(0)=u_1,~~\\ u|_{t \in (-t^*,0)} = \eta\in L^2(-t^*,0;H^2_0(\Omega)).
\end{cases}
\end{equation}
The delay potential $q(u^t,t)$ on the RHS is given by the function \newline $q :\, L^2(-t^*,0;H^2_0(\Omega))\times \R\mapsto \R$.
 The scalar $k_0> 0$ is a damping coefficient that includes imposed structural damping, and damping through the flow via Theorem \ref{rewrite}. The continuous, linear operator $L:H^{\sigma}(\Omega) \to L^2(\Omega),~~\sigma<2$ encompasses spatial lower order terms that need not have conservative structure
 (e.g., the term $-Uu_x$  in \eqref{reducedplate}).
 \end{remark}
 
 In application, we will consider an initial datum $y_0 \in Y$ corresponding to the dynamics $S_t(y_0)$ in \eqref{flowplate} (the full flow-plate dynamics). We employ the reduction result  Theorem \ref{rewrite}, and we may consider the ``initial time" ($t=t_0$) for the delay dynamics corresponding to any time after the reduction time $t^{\#}(\rho_0,U,\Omega)$ above. At such a time, the data which is fed into \eqref{reducedplate} is $x_0=(u(t_0),u_t(t_0),u^{t_0})$, where this data is determined by the full dynamics of \eqref{flowplate} on $(t_0-t^*,t_0)$. Thus, given a trajectory $S_t(y_0)=y(t)=(u(t),u_t(t);\phi(t),\phi_t(t)) \in Y$, we may analyze the corresponding delay evolution $(T_t,\mathbf H)$, with $\mathbf H \equiv H_0^2(\Omega)\times L^2(\Omega)\times L^2\left(-t^*,0;H_0^2(\Omega)\right),$ with given data $x_0 \in \mathbf H$. We then have that $T_t(x_0)=\left(u(t),u_t(t);u^t\right)$ with $x_0=(u_0,u_1,\eta)$. The norm is taken to be $$||(u,v;\eta)||^2_{\mathbf H} \equiv ||\Delta u||^2+||v||^2+\int_{-t^*}^0||\Delta \eta(t+s)||^2 ds.$$

	\subsection{Attractors}
	In this section we refer to the delay dynamical system $(T_t,\mathbf H)$ corresponding to the previous section. We emphasize that, for all of these results, we   need not impose any structural damping, i.e., $k$ can be taken to be zero since the damping is inherited from the flow via \eqref{reducedplate}.
		\begin{theorem}[Smooth, Finite Dimensional Global Attractor]\label{maintheorem}
Let $b_2,k\ge 0$, $U \neq 1$,  $p_0 \in L^2(\Omega)$, and $b_1 \in \mathbb R$ in \eqref{flowplate}. Also assume the flow data $\phi_0,\phi_1 \in Y$ are localized (as in Theorem \ref{rewrite}).  Then the corresponding delay system $(T_t,\bH)$ has a compact global attractor $\bA$ of finite fractal dimension. Moreover, $\bA$ has additional regularity: any full trajectory $y(t)=(u(t),u_t(t),u^t)\subset \bA$, $t\in\R$, has the property that $u \in L^{\infty}(\R;H^4(\Omega)\cap H_0^2(\Omega))$, $u_t \in L^{\infty}(\R;H_0^2(\Omega))$, and $u_{tt} \in L^{\infty}(\mathbb R; L^2(\Omega))$.
\end{theorem}
 This can be rephrased for the non-delay system $(S_t,Y)$, by taking the previous result with Theorem \ref{rewrite} and using projection onto the first two components of $\mathbf H$:
\begin{corollary}\label{th:main2*}
With the same hypotheses as Theorem \ref{maintheorem},  there exists a compact set $\mathscr{U} \subset H_0^2(\Omega) \times L^2(\Omega)$ of finite fractal dimension such that for any weak solution $(u,u_t;\phi,\phi_t)$ to (\ref{flowplate}) with
initial data
$
(u_0, u_1;\phi_0,\phi_1) \in Y
$ 
that have a localized flow component $supp(\phi_0),~supp(\phi_1) \subset K_{\rho_0}$ for some $\rho_0>0$: $$\lim_{t\to\infty} d_{Y_{pl}} \big( (u(t),u_t(t)),\mathscr U\big)=\lim_{t \to \infty}\inf_{(\nu_0,\nu_1) \in \mathscr U} \big( ||u(t)-\nu_0||_2^2+||u_t(t)-\nu_1||^2\big)=0.$$ We also have the additional regularity $\mathscr{U} \subset \big(H^4(\Omega)\cap H_0^2(\Omega)\big) \times H_0^2(\Omega)$.
\end{corollary}
Lastly, without imposing any damping, we have a generalized fractal exponential attractor:
\begin{theorem}[Generalized Fractal Exponential Attractor]\label{th:main2}
With the same hypotheses as Theorem \ref{maintheorem}, the evolution $(T_t,\mathbf H)$ has a generalized fractal exponential attractor ${A}_{\text{exp}} $ of finite dimension in the space $$\widetilde{\mathbf H}\equiv \widetilde{Y_{pl}}\times L^2(-t^*,0;L^2(\Omega)) = L^2(\Omega) \times H^{-2}(\Omega)\times L^2(-t^*,0;L^2(\Omega)).$$
\end{theorem}

	\subsection{Determining Functionals}
	
	Given a set of continuous, linear functionals $\mathscr L$ on $H_0^2(\Omega)$, the completeness defect of $\varepsilon_{\mathscr L}$ between $H_0^2(\Omega)$ and $L^2(\Omega)$ is defined by:
	\begin{equation}\label{defect}
	\varepsilon_{\mathscr L}(H_0^2(\Omega),L^2(\Omega)) \equiv  \sup_{\{||\Delta w|| \le 1\}}\big\{||w||_{L^2(\Omega)}~:~ l_j(w)=0~~~\forall~j=1,...,L\big\}.\end{equation}
	With this notion in hand, we can present our main theorem on determining functionals. 
	\begin{theorem}[Finite Number of Determining Functionals]\label{defecttheorem} Take the hypotheses from Theorem \ref{maintheorem} and consider $(T_t,\mathbf H)$ as above. Then there exists a number $\varepsilon_*>0$ such that if $\mathscr L$ is any set of functionals on $H_0^2(\Omega)$ with $\varepsilon_{\mathscr L}(H_0^2(\Omega),L^2(\Omega)) \le \varepsilon_*$, then $\mathscr L$ is a determining set of functionals for $(T_t,\bH)$. 
	\end{theorem}

	There is one situation where the completeness defect $\varepsilon_{\mathscr L}(H_0^2(\Omega),L^2(\Omega))$ can be estimated straightforwardly. For a given set of functionals $\mathscr L$ on $H_0^2(\Omega)$ and a given set of linearly independent functions $\{\phi_j\}_{j=1}^N\subseteq H_0^2(\Omega)$, define the interpolation operator $ R_{\mathscr L}:H_0^2(\Omega)\to H_0^2(\Omega)$, given by the formula
	$$R_{\mathscr L}(w)=\sum_{j=1}^N l_j(w)\phi_j.$$ 
	We say $R_{\mathscr L}$ approximates ``well" when there exists $C,\alpha>0$ such that
	$$||w-R_{\mathscr L}w||_{L^2(\Omega)} \le Ch^{\alpha},~~\forall ||\Delta w||\le 1$$
	for any $h(N)>0$ sufficiently small. It is immediate, then, that $\varepsilon_{\mathscr L} \le Ch^{\alpha}$; see \cite[Section 3.3]{quasi} for details.
	
	We now provide some concrete examples from the discussion above; this discussion is found in \cite{ch-0}, and see also \cite{quasi}. These examples, in conjunction with Theorem \ref{defecttheorem}, show that the structural dynamics have {\bf finite} determining {\em nodes, modes}, and {\em local volume averages}. Below, $c>0$ does not depend on $N$.
	\vskip.2cm
	\noindent{\bf Nodes}: Let $\mathcal T^h$ be a triangulation of $\Omega$ with triangles of side-length less than $h$; let $\{x_j~:~j=1,...,N_h\}$ be all vertices in $\mathcal T^h$. Then the set 
	$$\mathscr L=\{ l_j ~:~l_j(w)=w(x_j),~~j=1,...,N_h\}$$
	has completeness defect $\varepsilon_{\mathscr L}(H_0^2(\Omega),L^2(\Omega)) \le ch^2$.
	\vskip.2cm
	\noindent{\bf Modes}: Let $\{e_j\}$ be the eigenfunctions of the clamped biharmonic operator $\Delta^2$ acting on $H_0^2(\Omega)$. Then 
	the set
	$$\mathscr L=\{ l_j ~:~l_j(w)=(w,e_j)_{L^2(\Omega)},~~j=1,...,N\}$$
	has completeness defect $\varepsilon_{\mathscr L}(H_0^2(\Omega),L^2(\Omega)) \le c/N$.
	\vskip.2cm
	\noindent{\bf Averages}: Assume that $\lambda \in L^{\infty}(\mathbb R^2)$ with compact support and $\int_{\mathbb R^2}\lambda(\mathbf x) d\mathbf x =1$. For $h>0$, define
	$$\mathscr L=\{ l_j ~:~l_j(w)=\dfrac{1}{h^2}\int_{\Omega}w(\mathbf x)\lambda(\mathbf x/h-j)d\mathbf x,~~j=(j_1,j_2) \in \mathcal J\},$$
	where $\mathcal J\!\equiv\!\{(j_1,j_2)\!\in\!\mathbb Z^2\!:(j_1h,j_2h)\!\in\!\Omega \}$
	has completeness defect $\varepsilon_{\mathscr L}(H_0^2(\Omega),L^2(\Omega))\!\le ch^2$.

	\subsection{Discussion of Literature in Relation to Results}
	We begin by noting that a majority of works on flow-plate interactions consider the plate to be of scalar von Karman type. That nonlinearity is more formidable, and we direct the reader to the monograph \cite{springer} for discussion, as well as the papers \cite{supersonic,delay,HLW} for comparative discussion between von Karman and Berger dynamics plates. We consider $f=f_B$ primarily to clarify exposition, since the scalar von Karman system has additional technicalities that cloud the discussion; {\em our focus here is on the essential nature of finite dimensionality} in the system. We assert that {\em many of the results below hold for the von Karman system, perhaps adjusting the size and type of damping mechanism}---see \cite{fereisel,survey2}.
	
	We know that 	the nature of (in)stability for the dynamics depend critically on structural boundary conditions, and the flow parameter $U$. To expound the role of $U$, we note that the flow energy in \eqref{flowenergy} degenerates as $U \mapsto 1$; this necessitates treating the subsonic \cite{webster,jadea12} and supersonic cases independently \cite{supersonic}.  Long-time behavior results for the full flow-plate system with $U>1$ seem untenable, owing to the lack of a ``good" energy identity. On the other hand, invoking the reduction result Theorem \ref{rewrite}---only possible after establishing the results in \cite{supersonic}---allow one to consider all values of $U \neq 1$, so long as only the plate dynamics are considered. 

The earliest mathematical approaches to the flow-plate dynamics at hand invoke an ad hoc, piston-theoretic \cite{ch-0,springer} simplification ($q\equiv 0$ in \eqref{reducedplate}), or operate directly on the reduced system \cite{oldchueshov1,oldchueshov2} without first establishing the reduction in Theorem \ref{rewrite}. These works often provide well-posedness proofs, as well as constructions of compact global attractors, albeit without the quasi-stability technology utilized here. The more recent work \cite{HLW} investigates attractors and exponential attractors for a piston-theoretic model (both von Karman and Berger), using the modern quasi-stability framework.

Next, we note that all early works on the flow-plate dynamics or reduced dynamics typically utilize velocity regularization, i.e., some mechanism to boost the plate velocity $u_t \in L^2(\Omega) \to H^1(\Omega)$. Such an improvement has many benefits, discussed in more detail below. The most common means of doing this is through the inclusion of rotational inertia effects $\alpha>0$ (the so called {\em Rayleigh correction}):
	\begin{equation}\label{inertiaeq}
	(1-\alpha \Delta)u_{tt}+k(1-\alpha \Delta)u_t+\Delta^2u+f(u)=p(x,t)
	\end{equation}
	The damping mechanism above has been adjusted to reflect the strength of the inertial term (see \cite{HTW}). 
	
	The papers \cite{ryz,ryz2} do not consider inertia, but invoke plate thermoelasticity \cite{LT}, providing velocity regularization and dissipative effects.  In general, \cite{springer} provides a rather complete review for the above scenarios when $f(u)$ is given by the Berger or von Karman nonlinearity. {\em We stress that, even with beneficial thermal coupling or $\alpha>0$, flow-plate problems are still challenging due to the coupling, and reduced, delayed dynamics are challenging due to the intrinsically non-gradient character.}

\subsubsection{Previous Results with Velocity Regularization}\label{regularization}

	Well-posedness for flow-plate dynamics \eqref{flowplate} in past literature (before 2010) involved one of the aforementioned regularizations, resulting in $w_t \in H^1$. Here, one is still faced with the low regularity of traces (the failure of the Lopatinski condition \cite{miyatake}) for the Neumann wave equation. The $\alpha>0$ well-posedness method in \cite{LBC96,b-c-1,springer} relies on sharp microlocal estimates for the wave equation driven by $[w_t + U w_x]\in H^{1}(\Omega)$ Neumann data, yielding $r_{\Omega} tr[\phi_t] \in L^2\big(0,T; H^{-1/2}(\Omega)\big)$ \cite{springer}.  With an explicit 3-D wave solver, and a Galerkin approximation, one constructs a solution via a fixed point argument.  The method fundamentally decouples flow and plate dynamics, and limit passage on approximate solutions obtains through compactness of the Neumann lift, available when $\alpha>0$.

The primary physical model \cite{ciarlet,dowell1,lagnese} takes $\alpha=0$, but von Karman type nonlinearities {\em and} interactive flow terms do {\em not act compactly} in this case.  Moreover, for long-time behavior studies with $\alpha>0$, plate damping {must be tailored to}~ $-\alpha\partial_x^2w_{tt}$ to effectively control of kinetic energies. On the other hand, {\em frictional damping}, $kw_t$, is of the same strength as reduced flow damping in \eqref{reducedplate}, and thus there is a disparity between including rotational inertia $\alpha>0$ and aerodynamic and natural damping. The long-time behavior results in \cite{springer} are the most recent for the system \eqref{flowplate} with $\alpha>0$; these results include attractors, determining functionals, and subsonic convergence to equilibrium with imposed damping ($k>0$ in \eqref{inertiaeq}). In the references \cite{ryz,ryz2} $U<1$, a thermoelastic panel stabilizes {\em without additional mechanical damping}. {\em Again, both scenarios  rely on compactness of the boundary-to-flow map, conspicuously absent when $\alpha=0$.}

	\subsubsection{Previous Results for \eqref{flowplate} without Regularizations}
	We now turn to previous results for the system as presented in \eqref{flowplate}, with no regularizations. 
	
	In \cite{webster}, well-posedness of the $\alpha=0$ panel \eqref{flowplate} was established for $U<1$ using semigroup methods to treat the entire system. The approach is distinct from \cite{b-c-1,LBC96}, since corresponding component-wise estimates there are singular as $\alpha\searrow 0$. An alternative proof was given in \cite{jadea12}, where a viscosity approach through an absorbing boundary condition obtains solutions for \eqref{flowplate}. With the established viability of semigroup techniques for subsonic flows, the supersonic $U>1$ problem was recast in the abstract framework for the challenging $\alpha=0$, $U>1$ case. The reference \cite{supersonic} provides the well-posedness result, critically utilizing hyperbolic theory, where traces behave better than the standard theory predicts \cite{miyatake}. This recent resolution of well-posedness for all $U\neq 1$ with $\alpha=0$ opened the door to long-time behavior studies for \eqref{flowplate}---in particular, all of the Theorems \ref{maintheorem}--\ref{defecttheorem} here.

With the reduction result, Theorem \ref{rewrite}, established for the first time in \cite{delay} (on the strength of the well-posedness  above), one can also study end behaviors of the structural component of the system {\em without imposing any mechanical damping}. The paper \cite{delay} considers the reduced system in \eqref{genplatedelay} (with $\alpha=0$), with a nonlinearity of Berger, von Karman, or Kirchhoff type. The analysis provides the construction of compact global attractors that are smooth and finite dimensional. The proof we provide here is fundamentally different: by focusing on the Berger nonlinearity, we do not need to utilize the compactness of the attractor in order to obtain the quasi-stability estimate. 

Lastly, with the strong stabilization results for \eqref{flowplate} with $U<1$ (\cite{springer} for $\alpha>0$ and \cite{ryz,ryz2} for included thermal effects), the papers \cite{conequil1,conequil2} provide analogous results for Berger and von Karman plates. (Precise results depend on the nonlinearity in force, and the size of the damping $k>0$.)

\section{Technical Tools}\label{tools}
	
	\subsection{Dissipative Dynamical Systems}
	
		We recall notions and results from the theory of dissipative dynamical systems
 (see, e.g., \cite{Babin-Vishik,ch-0,lad,springer}).
 
We say
$(S_t,H)$ is \textit{asymptotically smooth} if for any
bounded, forward invariant set $D$ there exists a compact set $K
\subset \overline{D}$ such that $~
\lim_{t\to+\infty}d_{H}\{S_t D|K\}=0$. 
A closed set $B \subset H$ is \textit{absorbing} if for any bounded set $D \subset H$ there exists a $t_0(D)$ such that $S_tD \subset B$ for all $t > t_0$. If $(S_t,H)$ has a bounded absorbing set it is said to be \textit{ultimately dissipative}.

We will use a key theorem from \cite[Chapter 7]{springer} to establish the attractor and its characterization. 
\begin{theorem}
\label{dissmooth} A dissipative and asymptotically smooth dynamical system $(S_t,H)$ has a unique compact global attractor $\textbf{A} \subset H$ that is connected,  characterized by the set of all bounded, full trajectories.
\end{theorem}

	\subsection{Quasi-stability}\label{quasisec}
	
Quasi-stability is the primary tool in our long-time behavior analysis.  A quasi-stable dynamical system is one where {\em the difference of two trajectories can be decomposed} into uniformly stable and compact parts, with controlled scaling of powers. The theory of quasi-stable dynamical systems has been developed thoroughly in recent years by Chueshov and Lasiecka \cite{quasi,springer}, including more general definitions \cite{quasi} than what we present and use below. 

Informally, we mention that: 
\begin{itemize}
\item Having the quasi-stability property on the global attractor $\bA$ yields additional smoothness and finite dimensionality of $\bA$. This follows from the so called ``squeezing property" and one of Ladyzhenskaya's theorems (see \cite[Theorems 7.3.2 and 7.3.3]{springer}).
\item Having the quasi-stability estimate on an absorbing ball implies the existence of an exponentially attracting set; uniform in time H\"{o}lder continuity (in some topology) yields finite dimensionality of this attracting set in said topology. 
\end{itemize}

We now proceed more formally.
\begin{condition}\label{secondorder} Consider second order (in time) dynamics $(S_t,H)$ where $H=X \times Z$ with $X,Z$ Hilbert, and $X$ compactly embedded into $Z$.  Further, suppose $y= (x,z) \in H$ with $S_ty =(x(t),x_t(t))$ where the function $x \in C(\mathbb R_+,X)\cap C^1(\mathbb R_+,Z)$. 
\end{condition}
 Condition \ref{secondorder} restricts our attention to second order, hyperbolic-like evolutions.
\begin{condition}\label{locallylip} Suppose the evolution operator $S_t: H \to H$ is locally Lipschitz, with Lipschitz constant $a(t)\in L^{\infty}_{loc}([0,\infty))$:
\begin{equation}\label{specquasi}
||S_ty_1-S_ty_2||_H^2 \le a(t)||y_1-y_2||_H^2.
\end{equation}
\end{condition}
\begin{definition}\label{quasidef}
With Conditions \ref{secondorder} and \ref{locallylip} in force, suppose that the dynamics $(S_t,H)$ admit the following estimate for $y_1,y_2 \in B \subset H$:
\begin{equation}\label{specquasi*}
||S_ty_1-S_ty_2||_H^2 \le e^{-\gamma t}||y_1-y_2||_H^2+C_q\sup_{\tau \in [0,t]} ||x_1-x_2||^2_{Z_*}, ~~\text{ for some }~~\gamma, C_q>0,
\end{equation} where $Z \subseteq Z_* \subset X$, and the last embedding is compact. Then we say that $(S_t,H)$ is {\em quasi-stable} on $B$.
\end{definition}
\begin{remark}\label{genquas} As mentioned above, the definition of quasi-stability in the key references \cite{quasi,springer} is much more general; specifically, the estimate in \eqref{specquasi*} can be replaced with: \begin{equation}\label{genquaseq}
||S_ty_1-S_ty_2||_H^2 \le b(t)||y_1-y_2||_H^2+c(t)\sup_{\tau \in [0,t]} [\mu_H(S_ty_1-S_ty_2)]^2,  \end{equation} 
where: (i) $b(\cdot)$ and $c(\cdot)$ are nonnegative scalar functions on $\mathbb R_+$ such that $c(t)$ is locally bounded on $[0,\infty)$ and $b \in L^1(\mathbb R_+)$ and $\ds \lim_{t \to \infty} b(t) = 0$; (ii) $\mu_H$ is a compact seminorm on $H$. 
\end{remark}

We now run through a handful of consequences of the type of quasi-stability described by Definition \ref{quasidef} above for dynamical systems $(S_t,H)$ satisfying Condition \ref{secondorder}
\cite[Proposition 7.9.4]{springer}.
\begin{theorem}\label{doy}
If a dynamical system $(S_t,H)$ satisfying Conditions  \ref{secondorder} and \ref{locallylip} is quasi-stable on every bounded, forward invariant set $ B \subset H$, then $(S_t,H)$ is asymptotically smooth. Thus, if in addition, $(S_t,H)$ is ultimately dissipative, then by Theorem \ref{dissmooth} there exists a compact global attractor $\bA \subset \subset H$. 
\end{theorem}

The theorems in \cite[Theorem 7.9.6 and 7.9.8]{springer} provide the following result concerning improved properties of the attractor $\bA$ if the quasi-stability estimate can be shown {\em on} $\bA$. If Theorem \ref{doy} is used to construct the attractor, then Theorem \ref{dimsmooth} follows immediately; this is not always possible \cite{delay,HLW}.
\begin{theorem}\label{dimsmooth}
If a dynamical system $(S_t,H)$ satisfying Conditions  \ref{secondorder} and \ref{locallylip} possesses a compact global attractor $ \bA \subset \subset H$, and is quasi-stable on $\bA$, then $\bA$ has finite fractal dimension in $H$, i.e., $\text{dim}_f^H\bA <+\infty$. Moreover, any full trajectory $\{(x(t),x_t(t))~:~t \in \mathbb R\} \subset \bA$ has the property that
$$x_t \in L^{\infty}(\mathbb R;X)\cap C(\mathbb R;Z);~~x_{tt} \in L^{\infty}(\mathbb R;Z),$$ with bound
$$||x_t(t)||^2_X+||x_{tt}(t)||_Z^2 \le C,$$
where the constant $C$ above depends on the ``compactness constant" $C_q$ in \eqref{specquasi*}.
\end{theorem}
\noindent Elliptic regularity can then be applied to the equation itself generating the dynamics $(S_t,H)$ to recover regularity for $x(t)$ in a norm higher than that of the state space $X$.

The following theorem relates generalized fractal exponential attractors to the quasi-stability estimate \cite[p. 388, Theorem 7.9.9]{springer}:
\begin{theorem}\label{expattract*}
Let Conditions  \ref{secondorder} and \ref{locallylip} be in force. Assume that the dynamical system generated by solutions $(S_t,H)$ is ultimately dissipative and quasi-stable on a bounded absorbing set $ B$. Also assume there exists a space $\widetilde H \supset H$ so that $t \mapsto S_ty$ is H\"{o}lder continuous in $\widetilde H$ for every $y \in  B$; this is to say there exists $0<\alpha \le 1$ and $C_{ B,T>0}$ so that 
\begin{equation}\label{holder}||S_ty-S_sy||_{\widetilde H} \le C_{ B,T}|t-s|^{\alpha}, ~~t,s\in[0,T],~~y \in  B.\end{equation} Then the dynamical system $(S_t,H)$ possesses a generalized fractal exponential attractor $A_{\text{exp}}$ whose dimension is finite in the space $\widetilde H$, i.e., $\text{dim}_f^{\widetilde H} A_\text{exp}<+\infty$. 
\end{theorem}
\begin{remark}
We forgo using boldface to describe $A_{exp}$ (in contrast to {\bf the} global attractor $\bA$) precisely because exponential attractors are not unique. 
\end{remark}
\begin{remark}
In addition, owing to the abstract construction of the set $A_{\text{exp}} \subset X$, {\em boundedness} of $A_{\text{exp}}$ in any higher topology is not addressed by Theorem \ref{expattract*}. \end{remark}

The proofs of Theorems \ref{dimsmooth} and \ref{expattract*} can be found in \cite{quasi,springer}, and rely fundamentally on the technique of ``short" trajectories or ``l" trajectories (see, e.g.,  \cite{ltraj2}).

\section{Reduction to Delayed Dynamical System}\label{reducedsection}

In this section, we present the proof of Theorem \ref{rewrite}, which has multiple components. We remark that this theorem has been shown and used before, namely in \cite{delay}. We include the proof here because it relies critically on the well-posedness results for the system \eqref{flowplate} holding for all $U \neq 1$, and these are relatively recent. We note that analogous the result was shown earlier when velocity regularization (as discussed in Section 
ref{regularization}) was present---see \cite{LBC96,b-c-1,ryz,ryz2}.
	
	\subsection{Flow Potentials with Given Neumann Data; Decomposition}
	In what follows it will be necessary to consider the hyperbolic-like flow equation with {\em given} Neumann data. Consider the problem:
\begin{equation}\label{floweq*}
\begin{cases}
(\partial_t+U\partial_x)^2\phi=\Delta \phi & \text{ in }~\mathbb R_+^3\\
\partial_{z} \phi\Big|_{z=0} = h(\xb,t) & \text{ in }~\mathbb R^2\\
\phi(t_0)=\phi_0;~~\phi_t(t_0)=\phi_1
\end{cases}
\end{equation}
We have the following theorem from \cite{b-c-1,springer,miyatake}:
\begin{theorem}\label{flowpot}
Assume $U\ge 0$, $U\ne 1$; take $(\phi_0,\phi_1) \in H^1(\realsthree)\times L^2(\realsthree).$ If \newline$h \in C\left([t_0,\infty);H^{1/2}(\mathbb R^2)\right)$ then \eqref{floweq*} is well-posed (in the weak sense) with 
$$\phi \in C\left([t_0,\infty);H^1(\realsthree_+)\right),~~\phi_t \in C\left( [t_0,\infty);L^2(\realsthree_+)\right).$$
\end{theorem}

Now, we decompose the flow problem above:
\begin{equation}\label{flow1}
\begin{cases}
(\partial_t+U\partial_x)^2\phi^*=\Delta \phi^*  & ~\text{ in }~ \realsthree_+ \times (0,T) \\
\phi^*(0)=\phi_0; ~~\phi_t(0)=\phi_1 \\
\Dz \phi^* = 0 & ~\text { in } ~\partial \realsthree_+ \times (0,T)
\end{cases}
\end{equation}
\begin{equation}\label{flow2}
\begin{cases}
(\partial_t+U\partial_x)^2\phi^{**}=\Delta \phi^{**}  & ~\text{ in }~ \realsthree_+ \times (0,T) \\
\phi^{**}(0)=0; ~~\phi_t^{**}(0)=0 \\
\Dz \phi^{**} = h(\xb,t) & ~\text { in } ~\partial \realsthree_+ \times (0,T)
\end{cases}
\end{equation}
We focus on the case where the flow comes from the coupled system \eqref{flowplate}, so:
\begin{equation}\label{h}
h(\xb, t) \equiv [u_t + U u_x]_{ext}
\end{equation}
Then, the full flow solution $\phi$ coming from \eqref{flowplate} has the form $$\phi(t)=\phi^*(t)+\phi^{**}(t),$$ where $\phi^* (t) $ solves (\ref{flow1}) and $\phi^{**} (t)$ solves (\ref{flow2}), and $\phi^{**}$ depends on ~$u_t+Uu_x$ on $\Omega$.

\begin{remark}
In fact, a stronger regularity result is available. Finite energy ($H^1(\Omega) \times L^2(\Omega)$) solutions are obtained with $h\in H^{1/3}((0,T) \times \mathbb R^2)$ \cite{LT,tataru}, but the corresponding estimate doesn't provide control of $T$-dependence, and hence is of limited applicability in the long-time behavior context. 
\end{remark}

The analysis of $\phi^*$ is identical to that given in \cite{springer,ryz,b-c-1}. However, the treatment of 
$\phi^{**} (t)$, corresponding to the hyperbolic Neumann map,  is very different owing to the aforementioned loss of regularity. But here we are {\bf treating an a priori, existing finite energy solution}~ $(\phi,\phi_t) \in Y_{fl}$ ~corresponding  to Theorem \ref{nonlinearsolution}. As explained in Remark \ref{regularity},  treating the problem ``component-wise," as in \cite{b-c-1,Chu92b,ryz,ryz2}, would not be possible here.

We do have the following theorem from \cite{supersonic} concerning the a posteriori trace regularity of $\phi_t$ corresponding to a solution $(u,u_t;\phi,\phi_t)$ to \eqref{flowplate}, as in Theorem \ref{nonlinearsolution}:
\begin{theorem}
Let the hypotheses of Theorem \ref{maintheorem} be in force. Then a generalized solution $(u,u_t;\phi,\phi_t) \in C([0,T];Y)$ with flow component as in \eqref{floweq*} driven by $h=[u_t+Uu_x]_{ext}$ has the trace regularity:
 \begin{equation}\label{trace}
  ( \partial _{t} + U \partial_{x} ) tr [ \phi ] \in L^2(0, T; H^{-1/2}(\mathbb R^2)),
  \end{equation}
  where $T$ is arbitrary.
\end{theorem}

\begin{remark}\label{regularity}
With rotational inertia in force ($\alpha>0$ in \eqref{inertiaeq}), one would have for finite energy solutions  to \eqref{flowplate} ~$h=[u_t+Uu_x]_{ext} \in C([0, T]; H^1(\mathbb R^2))$.
On the other hand, from  \cite{miyatake}, $$h \in L^2(0,T;H^{1/2}(\realstwo) ) \mapsto
\phi^{**} \in  C([0,T];H^1(\realsthree_+) \cap C^1([0,T]; L^2(\realsthree_+)).$$ So the recovery of finite energy solutions is seen, and the Neumann mapping is in fact compact in this case.
When $\alpha=0$, one  has  only $h \in C([0,T]; L^2(\realstwo)) $, which produces a maximal regularity of  $$ \phi^{**} \in  C([0,T];H^{2/3}(\realsthree_+)) \cap C^1([0,T];H^{-1/3}(\realsthree_+)),$$ yielding  the loss of $1/3 $ derivative \cite{LT,delay,tataru}.  This underscores that the component-wise analysis of finite energy solutions to \eqref{flowplate}, successful  in  past literature
\cite{springer,ryz,ryz2}, cannot be utilized for $\alpha=0$.
  \end{remark}

  \subsection{Proof of Theorem \ref{rewrite}}
  
The analysis of $\phi^*$ uses classical tools. For the term $\phi^{**}$ we have the following theorem that provides us with an explicit form of the solution. The proof makes use of Fourier-Laplace transform methods; for a detailed proof of the representation, see for instance \cite[Theorem 6.6.10]{springer}.

Below, we utilize the notations: \begin{align} f^{\dag}(\xb,t,s,\theta)=&~f\left(x-\kappa_1(\theta,s,z), y-\kappa_2(\theta,s,z), t-s\right), \\ \kappa_1(\theta,s,z)=&~Us+\sqrt{s^2-z^2}\sin \theta,~~\kappa_2(\theta,s,z) = \sqrt{s^2-z^2}\cos \theta.\end{align}
\begin{theorem}\label{flowformula}
Considering the problem in \eqref{flow2} with  $h(\xb,t) =-[u_t (x,y,t)+Uu_x(x,y,t)]_{ext}$, there is a time $t^*(\Omega,U)$ such that we have the following representation for a weak solution $\phi^{**}(t)$ for $t>t^{*}$:
\begin{equation}\label{flowrep}
\phi^{**}(\xb,t) = -\dfrac{\chi(t-z) }{2\pi}\int_z^{t^*}\int_0^{2\pi}\Big([u_t]_{ext}^{\dag}(\xb,t,s,\theta)+U[u_x]_{ext}^{\dag}(\xb,t,s,\theta)\Big)d\theta ds.
\end{equation}
where $\chi(s) $ is the Heaviside function. The time $t^*$ is given by:
\begin{equation*}
   t^*=\inf \{ t~:~\xb(U,\theta, s) \notin \Omega \text{ for all } (x,y) \in \Omega, ~\theta \in [0,2\pi], \text{ and } s>t\},
\end{equation*} with ~$\xb(U,\theta,s) = (x-(U+\sin \theta)s,y-s\cos\theta) \subset \realstwo$.
\end{theorem}

Moreover, we have the following point-wise formula for the derivative in $t$ \cite[p. 480]{ryz2} (which is justified for smooth data in $\mathcal D(\mathbb T)$, and can be taken distributionally for data in $Y$). Differentiation of \eqref{flowrep} in $(x,y)$ is straightforward.
\begin{corollary}\label{calcs} Under the same hypotheses as Theorem \ref{flowformula}, we have:
\begin{align}\label{formderiv1}
\phi^{**}_t&(\xb, t) =  \dfrac{1}{2\pi}\Big\{\int_0^{2\pi}[u_t]_{ext}^{\dag}(\xb,t,t^*,\theta) d\theta -\int_0^{2\pi} [u_t]_{ext}^{\dag}(\xb,t,z,\theta )d\theta \\\nonumber
&+U\int_z^{t^*}\int_0^{2\pi}[\partial_x u_t]_{ext}^{\dag}(\xb,t,s,\theta) d\theta ds \\ &+\int_z^{t^*}\int_0^{2\pi}\dfrac{s}{\sqrt{s^2-z^2}}[M_{\theta}u_t]_{ext}^{\dag}(\xb,t,s,\theta) d\theta ds\Big\}  \nonumber
\end{align}
with $M_{\theta} = \sin(\theta)\partial_x+\cos(\theta) \partial_y$.
\end{corollary}
Therefore, to obtain the representation of ~$r_{\Omega}tr[\phi_t+U\phi_x]$ on the RHS of the plate equation in \eqref{flowplate}, we explicitly compute derivatives and restrict---this makes explicit the Neumann-to-Dirichlet map here. Indeed, 
the Kirchhoff representation for the solution  $\phi^*(\xb,t)$
in $\R_+^3$ (see, e.g.,  \cite[Theorem 6.6.12]{springer}), shows that, with $\phi_0$ and $\phi_1$  localized in $ K_{\rho}$,
then   Huygen's principle,
   gives  $\phi^*(\xb,t)\equiv 0$ for all $\xb\in  K_{\rho} $
and $t\ge t_{\rho}$. Thus we have that
\[
\big(\partial_t+U\partial_x\big)tr[\phi^*]\equiv0,~~~\xb\in \Omega,~t\ge t_{\rho}.
\]
Thus $\phi^*$ tends to zero in the sense of the local flow energy, i.e.,  \begin{equation}\label{starstable} \|\nabla \phi^*(t)\|_{L^2( K_{\rho} )}^2 + \|\phi^*_t(t)\|_{L^2( K_{\rho} )} \to 0, ~~ t \to \infty, \text{ for any fixed}~  \rho>0.\end{equation} 

It remains  to consider flow variable $\phi^{**}$, whose aeroelastic potential on the boundary coincides with
that of $\phi$, and hence it displays trace regularity as in (\ref{trace}) for $t>t_{\rho}$. This allows one to perform calculations with Corollary \ref{calcs}
on smooth solutions in order to obtain the representation
\begin{align*}
 ( \partial _{t} + U \partial_{x} )\gamma [ \phi^{**}] = & -\, h(x,y,t)
\\ &
+\, \frac{1}{2\pi}
 \int_{0} ^t   \int_{0}^{2\pi} 
[M_\theta h] (x - (U+\sin\theta)s,
y - s\cos\theta, t-s)~d\theta ds,
\end{align*}
for $h(x,t)=-[u_t+Uu_x]_{ext}$, yielding  Theorem \ref{rewrite}.

	\subsection{Estimates on the Delay Potential $q(u^t)$}\label{qest}

We now look at the structure of the delay potential $q(u^t)$ appearing in Theorem \eqref{rewrite}. 

\begin{lemma}\label{pr:q}
Let $q(u^t)$ be given by (\ref{potential}). Then \begin{equation}\label{qnegest}
||q(u^t)||^2_{H^{-1}(\Omega)} \le ct^*\int_{t-t^*}^t||u(\tau)||^2_{H^1(\Omega)}d\tau
\end{equation} for any $u \in L^2(t-t^*,t;H_0^1(\Omega))$.
If $u \in L^2_{loc}(-t^*,\infty;H^2\cap H_0^1)(\Omega))$ we also have
\begin{equation}\label{qnegest2}
||q(u^t)||^2 \le ct^*\int_{t-t^*}^t||\Delta u(\tau)||^2d\tau,~~~\forall t\ge0,
\end{equation}
and
\begin{equation}\label{qnegest3}
\int_0^t ||q(u^\tau)||^2 d\tau \le c[t^*]^2\int_{-t^*}^t||u(\tau)||^2_2d\tau
,~~~\forall t\ge0.
\end{equation}
Moreover, if $u \in C\left((-t^*,+\infty);(H^2\cap H_0^1)(\Omega)\right)$, we have that  \newline $q(u^t) \in C^1(\R_+; H^{-1}(\Omega))$, and for all $t \ge 0$
\begin{equation}\label{qnegest4}
\|\partial_t[q(u^t)]\|_{H^{-1}(\Omega)} \le C\Big\{ ||u(t)||_{H^1(\Omega)}+||u(t-t^*)||_{H^1(\Omega)}+\int_{-t^*}^0||\Delta u(t+\tau)||d\tau\Big\}.
\end{equation}
\end{lemma}
\begin{proof}
The proof of the bounds \eqref{qnegest}--\eqref{qnegest3} can be found in \cite{Chu92b} and \cite{springer}, and are straightforward.
Thus we need to check \eqref{qnegest4} only. Without loss of generality we can assume $u\in C\left((-t^*,\infty);
C_0^\infty(\Om)\right)$.
The following point-wise formula for the time derivative of $q(u^t)$ appearing above in \eqref{potential} is direct:
\begin{align}\label{qderiv}
\partial_t [q(u^t)] =&\int_0^{2\pi}\frac{1}{2 \pi}[M^2_{\theta} u]_{\text{ext}}\big(\xb(U,\theta,0),t \big)d\theta\\&-\int_0^{2\pi}\frac{1}{2 \pi}[M^2_{\theta} u]_{\text{ext}}\big(\xb(U,\theta,t^*),t-t^* \big)d\theta\nonumber\\
&+\int_0^{t^*}\int_0^{2\pi}(U+\sin\theta)\frac{1}{2 \pi}[M^2_{\theta} u_x]_{\text{ext}}\big(\xb(U,\theta,s),t-s \big)d\theta ds\nonumber\\
&+\int_0^{t^*}\int_0^{2\pi}(\cos\theta)\frac{1}{2 \pi}[M^2_{\theta} u_y]_{\text{ext}}\big(\xb(U,\theta,s),t-s \big)d\theta ds. \nonumber
\end{align}
\par
Recall that $\xb(U,\theta,s) = (x-(U+\sin\theta)s,y-s\cos\theta)$, and consider for any $\psi \in H^1_0(\Omega)$ the quantity $\big\langle \partial_t[q(u^t)],\psi \big\rangle$.
 In all the associated integrals, we extend the integration over $\Omega$ to all of $\realstwo$ and, recalling the definition of $M_{\theta}=\sin(\theta)\partial_x+\cos(\theta)\partial_y$, integrate by parts once in $M_{\theta}$ in the first and second terms, and integrate by parts once in space in the third and fourth terms. This leaves us with:
$$
|\langle \partial_t[q(u^t)],\psi\rangle | \le C\Big\{ ||u(t)||_1+||u(t-t^*)||_1+\int_{-t^*}^0||u(t+\tau)||_2d\tau\Big\}||\psi||_1.
$$
This implies the conclusion in \eqref{qnegest4}. (Detailed calculations are found in the Appendix of \cite{delay}.)
\end{proof}

	\subsection{General Nonlinear Plates with Delay}\label{delayplatesec}
	We now consider the delay system given through Theorem \ref{rewrite} as a standalone system. We have shown estimates corresponding to the delay potential in the previous section, and we now state some results for the general delay system. In analyzing the long-time behavior of the structural component $(u,u_t)$ of a trajectory $S_t(y_0)$ we will use multiplier methods on the system \eqref{reducedplate}.
	
	We use the notation as in Section \ref{reducedresults}. The parameter $0<t^*<+\infty$ is the time of delay and $u^t(\cdot)$,
 for a function on $s\in (-t^*,0)$, is of the form $s\mapsto u(t+s)$.
 We
need to impose an initial condition
 of the form $u|_{t \in (-t^*,0)} = \eta(\xb, t)$,
 where $\eta$ is a given function on $\Om \times (-t^*,0)$, specifically,
 $\eta \in L^2(-t^*,0;H^2_0(\Omega))$. 
 Thus we have the system
\begin{equation}\label{plate}\begin{cases}
u_{tt}+\Delta^2u+(k+1)u_t+f_B(u)= p_0-Uu_x+q(u^t) ~~ \text { in } ~\Omega\times (0,T), \\
u=\Dn u = 0  ~~\text{ on } ~ \partial\Omega\times (0,T),  \\
u(0)=u_0,~~u_t(0)=u_1,~~\\ u|_{t \in (-t^*,0)} = \eta\in L^2(-t^*,0;H^2_0(\Omega)).
 \end{cases}
\end{equation}
As mentioned in Section \ref{reducedresults}, the general delay plate equation \eqref{genplatedelay} can host a broad class of delay potentials, $q(u^t,t)$, for instance encompassing $q(u^t)$ is given in \eqref{potential}.
 The scalar $k\ge 0$ is an imposed damping coefficient, and represents structural weak damping across the full interior of the plate.  The operator term $Uu_x$ can be replaced by any spatial lower order terms which do not have gradient structure (as demarcated by $Lu$ in \eqref{genplatedelay}).
 
 Long-time behavior analysis of the delayed system depends on  the well-posedness of suitably
defined weak solutions which generate a dynamical system on the  phase space $\bH \equiv H_0^2(\Omega) \times L^2(\Omega) \times L^2(-t^*,0;H_0^2(\Omega))$.
 Well-posedness of weak solutions  has been addressed \cite{Chu92b} and
 \cite[Section 3.3.1]{springer}
via the Galerkin method, see also \cite{oldchueshov1,oldchueshov2}. In what follows we  summarize and complement relevant results.
\par

A \textit{weak solution} to \eqref{plate} on $[0,T]$ is a function $$u \in L^{\infty}(0,T;H_0^2(\Omega))\cap W^{1,\infty}(0,T;L^2(\Omega)) \cap L^2(-t^*,0;H_0^2(\Omega))$$ such that the variational relation corresponding to \eqref{plate} holds (see, e.g., \cite[(4.1.39), p.211]{springer}).
\begin{lemma}\label{p:well}
 Consider \eqref{plate} with $q(u^t)$ as in \eqref{potential} with initial data
$$
(u_0,u_1,\eta) \in\bH = H_0^2(\Omega)\times L^2(\Omega) \times L^2(-t^*,0;H_0^2(\Omega)).
$$
Then (\ref{plate}) has a unique weak solution on $[0,T]$ for any $T>0$. This solution belongs to
the class
$C\left(0,T;H_0^2(\Omega)\right)\cap C^1\left(0,T;L^2(\Omega)\right),
$
and satisfies the energy identity
\begin{equation}\label{energyrelation}
E_{pl}(t)+(k+1)\int_s^t ||u_t(\tau)||^2d\tau=E_{pl}(s)+
\int_s^t\langle q(u^{\tau})-Uu_x(\tau),u_t(\tau)\rangle d\tau
\end{equation}
where the expression $E_{pl}$ is as before in \eqref{plateenergy}.
\end{lemma}

Careful analysis of the estimates in Lemma \ref{pr:q} yield the estimates below.
 \begin{lemma}\label{le:q}
For $q(u^t)$ as in \eqref{potential}, we have $\forall \e>0$
\begin{align}\label{hidden1}
\Big|\int_0^t \langle q(u^{\tau}),u_t(\tau)\rangle d\tau\Big| \le&~   C\e^{-1} t^*  \int_{-t^*}^t||u(\tau)||_2^2d\tau +\e \int_0^t ||u_t(\tau)||^2d\tau,~~~0 \le t\le T,
\end{align}
for any  $u \in L^2(-t^*,T;H^2(\Omega))\cap H^1(0,T;L^2(\Omega))$.
Also, there exists $\eta_*>0$ such that for every $\epsilon>0$  we have the estimate:
\begin{align}\label{hidden2} \Big|\int_0^t \langle q(u^t,\tau),u_t(\tau)\rangle d\tau\Big| \le& ~ \epsilon\int_{-t^*}^t ||u(\tau)||_{2}^2d\tau +C(t^*,\e)\cdot(1+T)\sup_{[0,t]}||u(\tau)||^2_{2-\eta_*},
\end{align}
 for  any $u \in L^2(-t^*,T;H^2(\Omega))\cap C(0,T;H^{2-\eta}(\Omega))\cap C^1(0,T;L^2(\Omega))$.
 \end{lemma}

We can now introduce the operator
$ T_t\, : \bH\mapsto \bH$ by the formula
\begin{equation}\label{semigroup}
  T_t(u_0, u_1, \eta) \equiv (u(t), u_t(t), u^t),
\end{equation}
where $u(t)$ solves \eqref{plate}.
   Lemma \ref{p:well} implies the following conclusion:
 \begin{corollary}\label{co:generation}
 $T_t$ is a strongly continuous semigroup on $\mathbf{H}$.
 \end{corollary}
\begin{proof}
Strong continuity is stated in Lemma~\ref{p:well}.
 The semigroup property  follows from uniqueness.
Continuity with respect to initial data follows from the stronger Lipschitz property given below. \end{proof}
\begin{lemma}\label{l:lip}
Suppose $u^i(t)$ for $i=1,2$ are weak solutions to (\ref{plate}) with different initial data and $z=u^1-u^2$. Additionally, assume that
\begin{equation}\label{bnd-R}
||u_t^i(t)||^2+||\Delta u^i(t)||^2 \le R^2, ~i=1,2
\end{equation}
 for some $R>0$ and all $t \in [0,T]$. Then there exists $C>0$ and $a_R\equiv a_R(t^*)>0$  such that
\begin{align}\label{dynsysest} \nonumber
||z_t(t)||^2+||\Delta z(t)||^2 \le & \\ Ce^{a_Rt}\Big\{||\Delta(u^1_0&-u^2_0)||^2+||u^1_1-u^2_1||^2+\int_{-t^*}^0||\eta^1(\tau)-\eta^2(\tau)||_2^2d\tau\Big\}
\end{align}
for all $t \in [0,T]$.
\end{lemma}
We omit the details of these proofs here and refer to \cite{delay}. It suffices to say that energy methods are used, along with Lemma \ref{le:q} and an application of Gr\"{o}nwall.
We conclude this section with a remark about the case when additional velocity smoothing is present---namely, when rotary inertia or thermal effects are included in the model and $u_t \in H_0^1(\Omega)$.
\begin{remark}
{\rm
A priori, when $u_t$ is in $H^1_0(\Omega)$, it is clear from \eqref{qnegest} that \begin{equation}\label{thisone}
 \int_0^T\langle q(u^\tau),u_t(\tau)\rangle d\tau\le
 \epsilon \int_0^T ||u_t(\tau)||^2_1+C(\epsilon,T)\sup_{\tau \in [-t^*,T]}||u(\tau)||^2_{1}.
\end{equation}
 This is not at all apparent when $u_t \in L^2(\Omega)$ only, as $||q(u^t)||^2_0$ has no a priori bound from above like \eqref{qnegest}.
Hence, the critical component which allows us a transition from rotational inertia ($\alpha>0$) to the non-rotational case ($\alpha=0$) is the hidden compactness of the aforementioned term displayed by \eqref{qnegest4} obtained from integrating by parts under time integration in the LHS of \eqref{thisone}.
}
\end{remark}

\section{Supporting Technical Results}\label{techstuff}
	\subsection{Basic Estimates}
Consider the difference of two weak solutions $u^i$, $i=1,2$ to \eqref{plate}, satisfying:
\begin{equation}\label{difference}
\begin{cases}z_{tt}+\Delta ^2 z + (k+1) z_t+f_B(u^1)-f_B(u^2)=q(z^t)-Uz_x,
 \\ z=\Dn z = 0 \text{ on $\partial \Omega$ },
 \\ z(0)=z_0 \in H^2_0(\Omega), ~z_t(0)=z_1 \in L^2(\Omega), ~z|_{(-t^*,0)}\in L^2(-t^*,0;H_0^2(\Omega)).
 \end{cases}
\end{equation}
We take this equation with the notations:
\begin{equation}\label{Ez} z=u^1-u^2;~~ E_z(t) \equiv \dfrac{1}{2}\Big\{||\Delta z(t)||^2 + ||z_t(t)||^2\Big\};~~\cF(z)=f(u^1)-f(u^2).\end{equation}

We will utilize (in several places) a  key decomposition of the term $\langle \cF(z),z_t\rangle_{\Omega}$ for the Berger nonlinearity $f_B$. 
The results stated in the following theorem can be found in \cite{Memoires,gw}, though we provide the key details below:
\begin{theorem}\label{nonest}
Let $u^i \in \mathscr{B}_R(H^2_0(\Omega))$, $i=1,2$.
 Then  we have:
\begin{equation}\label{f-est-lip}
||f(u^1)-f(u^2)||_{-\delta}  \le C_{\delta}\Big(||u^1||_2,||u^2||_2\Big)||z||_{2-\delta} \le C(\delta,R)||z||_{2-\delta},~~\forall~\delta \in [0,1].
\end{equation}
In addition, for $u^1,u^2 \in
C((s,t);(H^2\cap H_0^1)(\Omega))\cap C^1((s,t);L^2(\Omega))$, then we have:
\begin{equation*}
\big\langle \cF(z),z_{t}\big\rangle_{\Omega} =\dfrac{1}{2}\frac{d}{dt}Q_1(z)+P_1(z)
\end{equation*}
where
\begin{equation*}
Q_1(z)=b_2||\nabla u^1||^2||\nabla z||^2-b_1||\nabla z||^2
\end{equation*}
and
\begin{equation}\label{p1}
P_1(z)=b_2\langle \Delta u^1,u^1_t\rangle||\nabla z||^2-b_2\big(||\nabla u^1||^2-||\nabla u^2||^2\big)\langle\Delta u^2,z_t\rangle.
\end{equation}
\end{theorem}
\begin{proof}
Letting $z=u^1-u^2$, and letting $B(u)=(b_1-b_2||\nabla u||^2)$, we note two facts immediately:
\begin{align}\label{dif}
B(u^1)\Delta u^1-B(u^2)\Delta u^2= & ~b_1\Delta z -b_2\big[||\nabla u^1||^2\Delta u^1-||\nabla u^2||\Delta u^2\big] \\\nonumber
=&~b_1\Delta z -b_2\big[||\nabla u^1||^2\Delta z+(||\nabla u^1||^2-||\nabla u^2||^2)\Delta u^2\big]\\[.2cm]
\label{difff}\left|~||\nabla u^1||^2-||\nabla u^2||^2~\right|=&~\Big|||\nabla u^1||-||\nabla u^2||\Big|\left(||\nabla u^1||+||\nabla u^2||\right)\\ \nonumber
\le &~\left(||\nabla u^1||+||\nabla u^2||\right)||\nabla u^1 - \nabla u^2||~\le ~C(R)||z||_1,
\end{align}
From here, note that 
\begin{align*}
|| \mathcal F(z) ||_{L^2(\Omega)} = &~ || B(u^1)\Delta u^1-B(u^2) \Delta u^2 || \\
 \le & ~b_1||\Delta z|| 
 \\ &+b_2\big|\big|[||\nabla u^1||^2\Delta u^1-||\nabla u^1||^2\Delta u^2+||\nabla u^1||^2\Delta u^2-||\nabla u^2||^2\Delta u^2]\big|\big|\\
 \le &~b_1||\Delta z||+b_2||\nabla u^1||^2||\Delta z||+||\Delta u^2||\big[||\nabla u^1||^2-||\nabla u^2||^2\big]\\
 \le & ~C(b_1,b_2,R)||z||_{H_0^2(\Omega)}.
\end{align*}
The result then follows for all $\delta \in [0,1]$ through transposition.

Now, for the decomposition, we have:
\begin{align*}
\langle \mathcal F(z),z_t\rangle =&~ b_1\langle \Delta z, z_t\rangle -b_2\langle ||\nabla u^1||^2\Delta z,z_t\rangle - b_2\left\langle \Delta u^2[||\nabla u^1||^2-||\nabla u^2||^2],z_t\right \rangle\\
=&~ \dfrac{1}{2}\dfrac{d}{dt}\Big[-b_1||\nabla z||^2+b_2||\nabla u^1||^2||\nabla z||^2 \Big]-\dfrac{b_2}{2}||\nabla z||^2 \dfrac{d}{dt}||\nabla u^1||^2\\
& -b_2 [||\nabla u^1||^2-||\nabla u^2||^2]\langle\Delta u^2,z_t \rangle.
\end{align*}
Above, we have freely integrated by parts (invoking the boundary conditions on each $u^1,u^2 \in H_0^2(\Omega)$).
The result follows via one more time differentiation and integration by parts.
\end{proof}

\begin{lemma}
Let $u^i \in C(0,T;H_0^2(\Omega))\cap C^1(0,T;L^2(\Omega)) \cap L^2(-t^*,T;H_0^2(\Omega))$ solve (\ref{plate}) with appropriate initial conditions on $[0,T]$ for $i=1,2$. Then the following estimate holds for all $\e >0$, for some $\eta > 0$, and $0 \le t \le T$:
\begin{align}\nonumber
\int_0^t \big(||\Delta u||^2-||u_t||^2\big)d\tau \le&~\epsilon+\epsilon \int_0^t||u||^2_2 d\tau + C\int_{-t^*}^0 ||u(\tau)||_2^2 d\tau \\&-\int_0^t\langle f(u),u\rangle d\tau + |\langle u_t(t), u(t)\rangle | +|\langle u_t(0), u(0)\rangle | \\
& + C(\epsilon,t^*,T)\sup_{\tau \in [0,t]} ||u(\tau)||^2_{2-\eta}.
\nonumber
\end{align}
Moreover, in the case where we are considering the difference $z=u^1-u^2$ of solutions solving (\ref{difference}) with $(u^i(t),u^i_t(t)) \in \mathscr B_R(Y_{pl})$ for all $t\in [0,T]$, we may utilize the estimates in Theorem~\ref{nonest}  to obtain
\begin{align}\label{zmult}
\int_s^t \big(||\Delta z||^2-||z_t||^2\big)d\tau
\le &~\epsilon \int_s^t ||z||_2^2 d\tau +C\int_{s-t^*}^t ||z(\tau)||_{2-\s}^2 d\tau  \notag \\
&+ C(\epsilon,T,R)\sup_{\tau \in [0,t]} ||z(\tau)||^2_{2-\eta}+ E_z(t)+E_z(s),
\end{align}
where $E_z(t)$ is given by \eqref{Ez}.
\end{lemma}

The final class of estimates we  need are energy estimates for the $(z,z_t)$ terms defined as the solution to (\ref{difference}). 
The energy relation on $[s,t]$ for $z$ in (\ref{difference}) is given by
\begin{align}\label{energyfordelay}
\Ez(t)+(k+1)\int_s^t ||z_t||^2 d\tau =& ~\Ez(s) -\int_s^t \langle \cF(z),z_t\rangle d\tau +\int_s^t\langle q(z^{\tau}),z_t(\tau)\rangle d\tau
\\ \notag
&
-U\int_s^t\langle z_x(\tau),z_t(\tau)\rangle d\tau
\end{align}

From the above two estimates, \eqref{zmult} and \eqref{energyfordelay}, making use of Young's inequality and Sobolev inequalities, we have for $0 \le s < t \le T$, some $\eta>0$, and all $\e >0$:
\begin{align*}
\Ez(t)+(k+1)\int_s^t ||z_t||^2 d\tau \le& ~\Ez(s) +C_{\epsilon,T}\sup_{\tau \in [s,t]}||z||^2_{2-\eta} +\epsilon\int_s^t\big(||z||_2^2+||z_t||^2\big)d\tau\\\nonumber &-\int_s^t \langle f(u^1)-f(u^2),z_t\rangle d\tau+\Big|\int_s^t\langle q(z^{\tau}),z_t(\tau)\rangle d\tau\Big|
\end{align*}
For all $k\ge 0$, \eqref{zmult} taken with the above
 implies that
\begin{align}\label{zenergyprelim-2}
\frac12 \Ez(t)+c_0\int_s^t E_z d\tau \le& ~\Ez(s) +C(T,R)\sup_{\tau \in [s,t]}||z||^2_{2-\eta}
+C\int_{s-t^*}^t ||z(\tau)||_{2-\s}^2 d\tau
\\\nonumber &-\int_s^t \langle f(u^1)-f(u^2),z_t\rangle d\tau+\Big|\int_s^t\langle \partial_t[q(z^{\tau})],z(\tau)\rangle d\tau\Big|
\\ \notag
&+ \big|\langle q(z^t),z(t)\rangle \big|+\big|\langle q(z^s),z(s)\rangle\big|.
\end{align}
Above, we integrated by parts in the integral with the delayed
term. Therefore there exist $a_i>0$ and $C(T,R)>0$ such that
\begin{align}\label{zenergyprelim-3}
\Ez(t)+\int_s^t E_z d\tau \le& ~ a_0\left(\Ez(s)+\int_{s-t^*}^s ||z(\tau)||_{2-\s}^2 d\tau\right)
  +C(T,R)\sup_{\tau \in [s,t]}||z||^2_{2-\eta_*}
\\\nonumber &-a_1\int_s^t \langle f(u^1)-f(u^2),z_t\rangle d\tau.
\end{align}
Taking $t=T$ and integrating over $s$ in $[0,T]$ we arrive at (possibly rescaling
$a_i$ and $C(T,R)$):
\begin{align*}
T\Ez(T)+ &\int_0^T \int_s^T E_z d\tau ds \le ~ a_0\left(\Ez(0)+\int_{-t^*}^0 ||z(\tau)||_{2-\s}^2 d\tau\right)
  +C_{T,R}\sup_{\tau \in [0,T]}||z||^2_{2-\eta_*}
\\\nonumber &-a_1\int_0^Tds \int_s^T \langle f(u^1)-f(u^2),z_t\rangle d\tau  -a_2\int_0^T \langle f(u^1)-f(u^2),z_t\rangle d\tau.
\end{align*}
Since
\[
\int_0^Tds \int_s^T E_z d\tau \ge \frac{T}2 \int_{T-t^*}^T E_z d\tau ~~~\mbox{for}~T\ge 2t^*,
\]
we arrive to the following assertion:

\begin{lemma}[Delayed Observability]\label{le:observbl}
Let $$u^i \in C(0,T;H_0^2(\Omega))\cap C^1(0,T;L^2(\Omega)) \cap L^2(-t^*,T;H_0^2(\Omega))$$ solve (\ref{plate}) with appropriate initial conditions on $[0,T]$ for $i=1,2$, $T\ge 2t^*$. Additionally, assume $(u^i(t),u^i_t(t)) \in \mathscr B_R(Y_{pl})$ for all $t\in [0,T]$.
Then the following observability estimate on $z$ holds:
\begin{align}\label{enest1}
 \frac{T}2 \Big[\Ez(T)+ \int_{T-t^*}^T \Ez(\tau) d\tau\Big] \le& ~ a_0\left(\Ez(0)+\int_{-t^*}^0 ||z(\tau)||_{2}^2 d\tau\right) \\\nonumber
 & +C(T,R)\sup_{\tau \in [0,T]}||z||^2_{2-\eta_*} \\\nonumber
  &-a_1\int_0^T \int_s^T \langle f(u^1)-f(u^2),z_t\rangle ~d\tau ds \\\nonumber & -a_2\int_0^T \langle f(u^1)-f(u^2),z_t\rangle d\tau,
\end{align}
 with the $a_i$ independent of $T$ and $R$.
\end{lemma}
	
	\subsection{Ultimate Dissipativity: Construction of an Absorbing Ball}\label{ball}
	In order to make use of Theorem \ref{dissmooth} (or any other abstract theorems presented in Section \ref{quasisec}), we must show that the non-gradient dynamical system $(T_t,\bH)$ is ultimately dissipative. To show this, we consider the delayed, Lyapunov-type function (with $E_{pl}$ as in  (\ref{plateenergy}) and with $\Pi_*(u)$ given by \eqref{posen}):
\begin{align}
V(T_t(x)) \equiv &~E_{pl}(u(t),u_t(t))-\langle q(u^t),u(t)\rangle+\nu\Big(\langle u_t,u\rangle +\frac{(1+k)}{2}||u||^2\Big)\\\nonumber
&+\mu\Big(\int_{t-t^*}^t\Pi_*(u(s))ds+\int_{0}^{t^*}\int_{t-s}^t \Pi_*(u(\tau))~d\tau ds\Big),
\end{align}
where $T_t(x) \equiv x(t)= (u(t),u_t(t),u^t)$ for $t \ge 0$\footnote{without loss of generality, take $t_0=0$} and $\mu, \nu$ are some positive numbers to be specified below. Recall
 the notation as in \eqref{posen}:
$$
E_* \equiv \dfrac{1}{2}[||\Delta u||^2+||u_t||^2]+\Pi_*(u).$$
From Lemma \ref{energybound} and the inequality
$$\int_0^{t^*}\int_{t-s}^t\Pi_*(u(\tau)) d\tau ds \le t^*\int_{t-t^*}^t \Pi_*(u(\tau)) d\tau,$$ we have that there exists a $\nu_0>0$ such that for all $0< \nu \le \nu_0$ there are $c_0(\nu),c_1,c(\nu),C >0$
\begin{equation}\label{energybounds}
c_0E_* - c \le V(T_t(x)) \le c_1E_*+\mu C t^*\int_{-t^*}^0 \Pi_*(u(t+\tau))d\tau +c.
\end{equation}

A careful but direct calculation of $\dfrac{d}{dt}V(T_t(x))$, coupled with the estimates on the nonlinear potential energy Lemma \ref{l:epsilon} and the estimate on $q(u^t)$ at the $L^2$ level in Lemma \ref{pr:q}, produces, 
for $0<\nu<\min{\{\nu_0,1\}}$, and for $\mu$ sufficiently small, the following lemma:
\begin{lemma}\label{le:48}
For all $k \ge 0$, there exist $\mu, \nu >0$ sufficiently small, and $c(\mu,\nu,t^*,k,b_2),$ $C(\mu,\nu,p_0,b_1,b_2)>0$ such that
\begin{align}\label{goodneg}
\dfrac{d}{dt}V(
T_t(x))\le& ~C-c\big\{||u_t||^2+||\Delta u||^2+||\Delta v(u)||^2 \\\nonumber &+\Pi_*(u(t-t^*))+\int_{-t^*}^0\Pi_*(u(t+\tau))d\tau\big\}.
\end{align}
\end{lemma}

From this lemma and the upper bound in \eqref{energybounds}, we have for some $\delta>0$  (again, depending on $\mu$ and $\nu$) and a $C$ (independent of the damping coefficient $k$): \begin{equation}\label{gronish}
\dfrac{d}{dt}V(T_t(x)) +\delta V(T_t(x)) \le C,~~t>0.
\end{equation}
The estimate above in (\ref{gronish}) implies (via an integrating factor) that
\begin{equation}\label{balll}
V(T_t(x)) \le V(x)e^{-\delta t}+\dfrac{C}{\delta}(1-e^{-\delta t}).
\end{equation}
Hence, the set
$$
\mathscr{B} \equiv \left\{x \in \bH:~V(x) \le 1+\dfrac{C}{\delta} \right\},
$$  is a bounded forward invariant absorbing set. This gives that $(T_t,\bH)$ is ultimately dissipative.

	\subsection{Quasi-stability on the Absorbing Ball}
	We adopt the approach here of showing the quasi-stability estimate \eqref{specquasi*} on the absorbing ball constructed in the previous section. 
	\begin{remark}We note that for other non-dissipative flow-plate systems (for instance involving the von Karman nonlinearity), the approach may differ; indeed, it is not always possible to show quasi-stability on the absorbing ball \cite{springer,HLW}---this is a rather strong property. 
	\end{remark}
Here, quasi-stability will follow directly from the observability inequality \eqref{enest1}, the nonlinear decomposition Theorem \eqref{nonest}, and the absorbing bound \eqref{balll}. In fact, the proof below demonstrates the quasi-stability estimate on any bounded, forward invariant set. 

Consider the decomposition in Theorem \eqref{nonest}:
\begin{align*}
\langle \mathcal F(z),z_t\rangle =&~ \dfrac{1}{2}\dfrac{d}{dt}\Big[-b_1||\nabla z||^2+b_2||\nabla u^1||^2||\nabla z||^2 \Big]+{b_2}||\nabla z||^2 \langle \Delta u^1,u^1_t\rangle\\
& -b_2 [||\nabla u^1||^2-||\nabla u^2||^2]\langle\Delta u^2,z_t \rangle.
\end{align*}

\noindent At this point, restricting to any bounded, forward-invariant set $B_R$ (radius denoted by $R$)\begin{equation*}
||u^1(t)||_{2}+||u^1_{t}(t)||_{0}+||u^2(t)||_{2}+||u^2_{t}(t)||_{0}\leq
C(R),~~t>0,
\end{equation*}
and taking into the Lipschitz-type bound \eqref{difff}, it follows immediately that, for  $0<\eta<1/2$:
\begin{align}\label{mmmbop}
\Big|\int_s^t \big\langle \cF(z),z_t\big\rangle_{\Omega} d\tau\Big| \le &~C(R,\epsilon)\sup_{\tau \in [s,t]} ||z||^2_{2-\eta}+ \epsilon\int_s^tE_{z}(t)d\tau,~~\forall~~\epsilon>0,
\end{align} provided
 $u^i(\tau) \in \mathscr{B}_R(H^2_0(\Omega))$ for all $\tau\in [s,t].$ In particular, this bound holds on the invariant, absorbing ball $\mathscr B$ from Section \ref{ball}.

Considering \eqref{enest1}, and taking $T$ sufficiently large, we have
from the observability inequality that:
\begin{equation*}
E_{z}(T)+\int_{T-t^*}^{T}||z(\tau)||_2^2 d\tau\leq \alpha \Big(E_z(0)+\int_{-t^*}^0||z(\tau)||_2^2d\tau\Big)+C\underset{\tau \in \lbrack0,T]}{\sup }||z(\tau )||_{2-\eta }^2
\end{equation*}
with $\alpha<1$ and $C=C(\mathscr{B},T,k,t^*)$.
By the standard iteration argument via the semigroup property, we conclude that  
\begin{equation}\label{qs*}
||\mathbf z(t)||_{\bH}^2 \le C(\sigma,\mathscr B)e^{-\sigma t}||\mathbf z(0)||_{\bH}^2+C(\mathscr B,t^*,k) \sup_{\tau \in [0,t]} ||z(\tau)||_{2-\eta}^2,
\end{equation}
for $\mathbf z(t)=(z(t),z_t(t),z^t)$, and thus $(T_t,\bH)$ is quasi-stable on $\mathscr B$.

\section{Proofs of Main Theorems}\label{proofs}

	\subsection{Proof of Theorem \ref{maintheorem}: Global Compact Attractor}
	On the strength of Theorem \ref{doy}, applied
with $B =\mathscr B$
and $
\bH = H_0^2(\Omega) \times
L^2(\Omega) \times L^2(-t^*,0;H_0^2(\Omega)),
$ we
deduce the existence of a compact global attractor from the quasi-stability property of $(T_t,\bH)$ given by \eqref{qs*} and \eqref{l:lip}. In addition, Theorem \ref{dimsmooth} guarantees $\mathbf{A}$ has finite fractal dimension and that $$||u_{tt}(t)||^2+||u_t(t)||_2^2 \le C~\text{ for all } t \in \R.$$
Since $u_{t}\in H^{2}(\Omega )\subset C(\overline \Omega)$, standard elliptic regularity with clamped boundary conditions for
$$
\Delta ^{2}u=p_0-u_{tt}-(1+k) u_t-f(u)-Uu_x+q(u^t) \in L^2(\Omega)
$$
 gives that ~$\ ||u(t)||_{4}^{2}\leq C ~~~\text{for all}~ ~
t\in \reals.$
Thus, we can conclude additional regularity of the trajectories from the attractor $\bA \subset \bH$
stated in Theorem~\ref{maintheorem}.
We have now completed the proof of Theorem \ref{maintheorem}.
	
Corollary \ref{th:main2*} follows immediately by considering the dynamical system for the full flow-plate system $(S_t,Y)$ that generates the reduced dynamical system  $(T_t, \bH )$ (possible for sufficiently large times by Theorem \ref{rewrite}). For $\mathbf A \subset \bH$, we then take $\mathscr U$ to be the projection of $\mathbf{A}$ on $Y_{pl}$, concluding the proof.

	\subsection{Proof of Theorem \ref{th:main2}: Generalized Fractal Exponential Attractor}
	With the quasi-stability estimate established on the absorbing ball, we need only establish the H\"{o}lder continuity in time of $T_t$ in {\em some} weaker space $\widetilde \bH$ to finish the proof of Theorem \ref{th:main2}. This is accomplished through lifting via the operator $\mathcal A^{-1/2}$ for $\mathcal A$ the clamped, biharmonic operator with domain $\mathcal D(\mathcal A) = (H^4\cap H_0^2)(\Omega)$ and $\mathcal Au=(-\Delta)^2u$. Via the standard construction, we have that for $u \in L^2(\Omega)$, we obtain $\mathcal A^{-1/2}u \in H_0^2(\Omega) = \mathcal D(\mathcal A^{1/2})$ \cite{springer,LT}.

From the previous section, we note that we can restrict our attention to the absorbing ball (for $t>t(x(0))$): $||x(t)||_{\bH} \le C(\mathscr B)$. In particular, for any $x(t)=(u(t),u_t(t),u^t)$, $t$ sufficiently large,  we have global-in-time bounds: 
\begin{align}
||\Delta u(t)|| \le C(\mathscr B),~ &~~||u_t(t)|| \le C(\mathscr B),~~~||u^t||_{L^2(-t^*,0;H_0^2(\Omega))} \le C(\mathscr B,t^*).
\end{align}
The latter follows from the dissipativity estimate in \eqref{energybounds} and the global-in-time bound of $V(T_t(x))$:
$$E_*(t) \le \frac{1}{c_0}\Big[V(T_t(x))+c\Big] \le C(\mathscr B).$$
And thus we have from the equation \eqref{reducedplate} and linearity of $q(\cdot)$ \begin{align} \mathcal A^{-1/2} u_{tt}=&~\mathcal A^{1/2}u+q(\mathcal A^{-1/2} u^t) +\mathcal A^{-1/2}\Big[p_0-U u_x -(k+1)  u_t - f_B(u)\Big].
 \end{align}
 From this it follows that
 \begin{align*} ||\mathcal A^{-1/2}u_{tt}||_{L^2(\Omega)} \le&~ C ||\Delta u||+t^*\int_{t-t^*}^t||\mathcal A^{-1/2}\Delta u||ds\\
 &+C||p_0||_{H^{-2}(\Omega)}+C(U)||u||_{H^{-1}(\Omega)}+C(k)||u_t||_{H^{-1}(\Omega)} \\
 \le &~C(t^*,U)||u||_{H^2(\Omega)}+C(k)||u_t||_{L^2(\Omega)}+C||p_0||_{L^2(\Omega)}\\
 \le & ~ C(t^*,U,p_0,\mathscr B).
 \end{align*}
From here, we note
$u_t(t)-u_t(s) =\int_s^tu_{tt}(\tau)d\tau,$ and thus
\begin{align} \nonumber
||u_t(t)-u_t(s)||_{H^{-2}(\Omega)} \le&~ C||\mathcal A^{-1/2} [u_t(t)-u_t(s)]||_{L^2(\Omega)}\\  \le & ~ \int_s^t ||\mathcal A^{-1/2}u_{tt}(\tau)||d\tau  \le C(t^*,U,k,\mathscr B)|t-s|.
\end{align}
Lastly, we note
\begin{align*}
||u(t)-u(s)|| \le ~ \int_s^t ||u_t(\tau)|| d\tau&  \le \Big(\sup_{t}||u_t||\Big)|t-s|\le C(\mathscr B)|t-s|, \\
\int_{-t^*}^0||u(t+\tau)-u(s+\tau)||d\tau \le&~ \int_{-t^*}^0 \int_{s+\tau}^{t+\tau}||u_t(\sigma)||d\sigma d\tau \le C(t^*,\mathscr B)|t-s|.
\end{align*}
From the above estimates, we see that 
$$||T_t(x)-T_s(x)||_{\widetilde{\bH}} \le \mathcal C|t-s|,~~\widetilde{\bH}=L^2(\Omega) \times H^{-2}(\Omega) \times L^2(-t^*,0;L^2(\Omega)).$$

Thus we note that $(T_t,\bH)$ is uniformly H\"{o}lder continuous (in fact, Lipschitz) on the absorbing ball $\mathscr B$ in the topology  $\widetilde{\bH} =\widetilde{Y_{pl}}\times L^2(-t^*,0;L^2(\Omega))$.
The proof of Theorem \ref{th:main2} is concluded on the strength of Theorem \ref{expattract*}.

\subsection{Proof of Theorem \ref{defecttheorem}: Construction of Determining Functionals}\label{construct}

In this proof we adapt \cite[Section 7.9.4]{springer} to show that having quasi-stability estimate for the dynamics $(T_t,\bH)$ on $\mathscr B$ is sufficient to produce a finite set of determining functionals (of sufficiently small completeness defect). 

For this proof, let $\mathscr L=\{l_i\}_{i=1}^N$ be a finite set of functionals on $H_0^2(\Omega)$. Recall the notion of {\em completeness defect} for $\mathscr L$ \eqref{defect} measured between $H_0^2(\Omega)$ and any $H^s(\Omega)$ ($0 \le s <2$):
	\begin{equation}
	\varepsilon_{\mathscr L}(H_0^2(\Omega),H^s(\Omega))=\varepsilon_{\mathscr L,s} \equiv  \sup_{\{|| \Delta w|| \le 1\}}\big\{||w||_{H^s(\Omega)}~:~ l_j(w)=0~~~\forall~j=1,...,N\big\}.\end{equation} 
	Secondly, we note the relation between $\varepsilon_{\mathscr L,2-\eta}$ and $\varepsilon_{\mathscr L,0}$ through Sobolev interpolation \cite[p.123]{quasi}:
	\begin{equation}\label{control} ||u||_{2-\eta} \le ||u||^{\eta/2}||u||_{2}^{1-\eta/2}~~\implies~~[\varepsilon_{\mathscr L,2-\eta}]^{2/\eta}\le  C(\eta)\varepsilon_{\mathscr L,0}.\end{equation}
	
	Now, let us prove a critical lemma.
	\begin{lemma}\label{doya} Let $\mathscr L$ and $\varepsilon_{\mathscr L,2-\eta}$ for $\eta>0$ as above. Then, for $v \in H^2_0(\Omega)$
	\begin{equation}\label{lemmause}||v||_{2-\eta} \le \varepsilon_{\mathscr L,2-\eta}||v||_{H^2_0(\Omega)}+C({\mathscr L},\eta)\max_{j = 1,...,N} |l_j(v)|.\end{equation}
	\end{lemma}
	\begin{proof}[Proof of Lemma \ref{doya}]
	Let $\{ e_j~:~j=1,...,N\}$ be an orthonormal system for $\mathscr L$---so $l_j(e_i)=0,~i\neq j$ and $l_j(e_j)=1$. Now, for any $v \in H^2_0(\Omega)$, we can write
	$w \equiv v-\sum_{j=1}^N l_j(v)e_j$, and this $w$ has the property that $l_j(w)=0$ for $j=1,...,N$. By the definition of $\varepsilon_{\mathscr L,2-\eta}$, we have
	$$||w||_{2-\eta} \le \varepsilon_{\mathscr L,2-\eta}||w||_{H_0^2(\Omega)}.$$ Substituting $v=w+\sum_{j=1}^Nl_j(v)e_j$, we obtain \eqref{lemmause}.
	\end{proof}

	Let $T_t(x_1)$ and $T_t(x_2)$ be two trajectories for $x_1,x_2 \in \mathscr B\subseteq  \bH$ (and let us retain the notation that $T_t(x_1)-T_t(x_2) = \mathbf z(t) = (z(t),z_t(t),z^t)$).
We want to show, if $\varepsilon_{\mathscr L,2-\eta}$ is sufficiently small, then:  \begin{equation}\label{limlim}\lim_{t \to \infty} | l_j(T_t(x_1)-T_t(x_2)))|^2 = 0,~~\forall~~j=1,...,N,\end{equation}
implies that
 $$\lim_{t \to \infty} ||T_t(x_1)-T_t(x_2)||_{H_0^2(\Omega)}^2 =0.$$
 So, suppose \eqref{limlim}. Note that this convergence is equivalent to 
 \begin{equation}\label{needthis}\mathscr S(t) \equiv \sup_{s \in [t,t+\tau]} \max_j |l_j(u^1(s)-u^2(s))|^2=0,~~t\to \infty.\end{equation}
 From the quasi-stability estimate \eqref{qs*} and the semigroup property, we obtain
 \begin{equation}\label{usingthisone}
 ||T_{t+\tau}(x_1)-T_{t+\tau}(x_2)||_{\bH}^2 \le C(\sigma,\mathscr B)e^{-\sigma \tau}||T_t(x_1)-T_t(x_2)||_{\mathbf H}+C\sup_{t \le s \le t+\tau} ||z(t)||_{2-\eta}^2
 \end{equation}
 With Young's inequality, we have from \eqref{lemmause}
 $$||v||_{2-\eta}^2 \le (1+\delta)\varepsilon_{\mathscr L,2-\eta}^2||v||_{H_0^2(\Omega)}^2+C(\mathscr L,\delta,\eta)\max_{j=1,...,N} |l_j(v)|^2.$$
 With the Lipschitz estimate on $T_t$ in \eqref{dynsysest}, we obtain from above
 $$\sup_{t \le s \le t+\tau} ||z(s)||_{2-\eta}^2 \le [(1+\delta)\varepsilon_{\mathscr L,2-\eta}Ce^{a_R\tau}]||T_t(x_1)-T_t(x_2)||_{\bH}^2+C(\mathscr L,\delta,\eta)\mathscr S(t).$$
 From this estimate, we invoke \eqref{usingthisone} to obtain
 $$||T_{t+\tau}(x_1)-T_{t+\tau}(x_2)||_{\bH}^2 \le \beta ||T_t(x_1)-T_t(x_2)||_{\bH}^2 + C(\mathscr L,\delta,\eta)\mathscr S(t),$$
 with $\beta=\mathscr C(\sigma,\mathscr B)[(1+\delta)\varepsilon_{\mathscr L,2-\eta}e^{a_R\tau}+e^{-\sigma \tau}]$. For $\delta>0$, and $\tau >0$ sufficiently large, by taking $\varepsilon_{\mathscr L,2-\eta}<\varepsilon_*$ sufficiently small, we guarantee  $\beta<1$. Then, again from the semigroup property, we can iterate on intervals of size $\tau$ to obtain
 \begin{align*} ||T_{t_0+n\tau}(x_1)-T_{t_0+n\tau}(x_2)||_{\bH}^2 \le &~  \beta^n||T_{t_0}(x_1)-T_{t_0}(x_2)||_{\bH}^2 \\ &+C\sum_{m=0}^{n-1}\beta^{n-m-1}\mathscr S(t_0+m\tau).\end{align*}
 From here, taking $n\to \infty$, we obtain from \eqref{needthis} the desired conclusion in \eqref{limlim} and the proof of Theorem \ref{defecttheorem} is complete, noting that $\varepsilon_{\mathscr L,0}$ controls $\varepsilon_{\mathscr L,2-\eta}$,  as in \eqref{control}.

\section{Final Comments}\label{finale}
In this concluding section, we make a few remarks. The main item is: {\em what happens in the system when there is imposed damping}, $k>0$? We quote a variety of results (whose proofs are beyond the scope here) that apply to our main models of interest \eqref{flowplate} (the full flow-plate system) and \eqref{reducedplate} (the reduced plate). We then propose some open questions regarding the results and analysis here.

	\subsection{Known Results for Imposed Damping}
	In this section we allow the imposed damping $k>0$, which leads to dissipation in the full flow-plate system \eqref{flowplate}, and additional damping in the reduced, delayed plate system \eqref{reducedplate}. We make the distinction here between {\em some damping}---$k>0$---and {\em large damping}---$k>k_*$ for $k_*$ chosen based on intrinsic properties of the model. In the results below, we need {\em large damping}. 
	
	\begin{remark}\label{kmin} In this result, and all other results below, the minimal damping coefficient $k_{*}$ depends on the loading $p_0$ and $b_1,b_2$, as well as the domain $\Omega$, the flow support parameter $\rho_0$, and the unperturbed flow velocity $U$, but $k_*$ is independent of the particular initial data of the system. \end{remark}
	
	The first result concerns the improvement of the generalized fractal exponential attractor in Theorem \ref{th:main2}. Namely, there exists a $k_*$ such that for all $k>k_*$, a proper fractal exponential attractor exists (also called an {\em inertial set} \cite{springer,quasi}). This set is exponentially attracting and finite dimensional in the state space.
	\begin{theorem}[Fractal Exponential Attractor]\label{th:main3}
With the same hypotheses as Theorem \ref{maintheorem} and $k>k_*$ (depending on the intrinsic parameters in \eqref{flowplate}) the evolution $(T_t,\mathbf H)$ has a fractal exponential attractor ${\mathbf A}_{\text{exp}} $ of finite dimension in the space ${\mathbf H}$.
\end{theorem}
The improvement uses the recent criterion by Chueshov in \cite{delay}, which itself makes use of the transitivity of exponential attraction described in \cite{fgmz}. The proof of the above theorem is a simple adaptation of the argument found in \cite{HLW}.
	
	The next known result concerns the entire flow-plate system, \eqref{flowplate}. A major hurdle in the long-time behavior analysis of the system is the transference of stability properties of the plate back to the (hyperbolic) flow through the Neumann mapping; at present, when $\alpha=0$, this is only possible in the subsonic case $U<1$, when the flow equation is truly hyperbolic (a perturbed wave equation). In this scenario, we have a ``good" energy relation, and the presence of damping with the energy relation provide {\em finiteness of the dissipation integral}. In this case, sufficiently large damping $k>k_*$ is enough provide convergence to equilibrium for the entire flow-plate trajectories. From a physical point of view, this says that flutter is excluded as an end behavior when the flow is subsonic. This is a well-known phenomenon to aeroelasticians: subsonic panels do not flutter. 

To state this result precisely, we present the stationary problem associated to \eqref{flowplate}, of the form:
\begin{equation}\label{static}
\begin{cases}
\Delta^2u+f_B(u)=p_0(\xb)+Ur_{\Omega}tr[\partial_x \phi]& \xb \in \Omega\\
u=\Dn u= 0 & \xb \in \Gamma\\
\Delta \phi -U^2 \partial_x^2\phi=0 & \xb \in \realsthree_+\\
\Dz \phi = U\partial_x u_{\text{ext}}  & \xb \in \partial \realsthree_+
\end{cases}
\end{equation}
The following theorem is shown for subsonic flows (this is given as \cite[Theorem 6.5.10]{springer}):
\begin{theorem}\label{statictheorem}
Suppose $0 \le U <1$, $ k \geq 0$, with $p_0 \in L^2(\Omega)$. Then {\em weak} solutions $\left(u(\xb),\phi(\xb)\right)$ to \eqref{static} exist and satisfy the additional regularity property $$(u,\phi) \in (H^4\cap H_0^2)(\Omega) \times W^2(\realsthree_+).$$
\end{theorem} \noindent
 {\em We denote the set of all stationary solutions (weak solutions to \eqref{static} above) as $\mathcal N$,} that is
 $$\mathcal N \equiv \{(\hat u,\hat \phi) \in H_0^2(\Omega) \times W^1(\realsthree_+): (\hat u,\hat \phi) ~\text{satisfy \eqref{static} variationally}\}.$$

Then we have the following theorem for the entire flow-plate system, when $k$ is large and $0\le U<1$:
\begin{theorem}\label{contoequil}
Let $0\le U<1$ and assume $p_0 \in L^2(\Omega)$ and $b_1 \in \mathbb R$. Assume $y_0=(u_0,u_1;\phi_0,\phi_1) \in Y$. Then there is a minimal damping coefficient $k_{*}>0$ (not depending on the particular solution) so that for $k\ge k_{*}>0$ any generalized solution $(u(t),\phi(t))$ to the system with localized initial flow data (i.e.,   $supp(\phi_0),supp(\phi_1)\subset K_{\rho_0}$ for some $\rho_0>0$) has the property that for any $\rho>0$ \begin{align*}\lim_{t \to \infty} \inf_{(\hat u,\hat \phi) \in \mathcal N}\Big\{\|u(t)-\hat u\|^2_{H^2(\Omega)}+\|u_t(t)\|^2_{L^2(\Omega)}&\\+\|\phi(t)-\hat \phi\|_{H^1( K_{\rho} )}^2+&\|\phi_t(t)\|^2_{L^2( K_{\rho} )} \Big\}=0.\end{align*} \end{theorem}

\subsection{Open Questions and Conjectures}
Let us now provide a few conjectures/open questions for further research along the lines in this paper. 
\vskip.1cm
\noindent{\bf Improving the Exponential Attractor Further}: As commented on in the conclusion of \cite{HTW}, it seems that the decomposition presented there is viable for the delay system. Indeed, with large damping and finiteness of the dissipation integral, we make the following conjecture:
	\begin{conjecture}[Fractal Exponential Attractor]
With the same hypotheses as Theorem \ref{maintheorem} and $k>k_*$ (depending on the intrinsic parameters in \eqref{flowplate}) the evolution $(T_t,\mathbf H)$ has a fractal exponential attractor (as in Theorem \ref{th:main3}) and ${\mathbf A}_{\text{exp}}\subset (H^4\cap H_0^2)(\Omega)\times H_0^2(\Omega)\times L^2(-t^*,0;(H^4\cap H_0^2(\Omega)))$---bounded in that topology.
\end{conjecture}

\vskip.1cm
\noindent{\bf Determining Functionals for the Entire System}: We note that in the reference \cite{springer} the determining functionals produced for the system $(T_t,\bH)$ are extended to a set of determining functionals for the entire dynamics $(S_t,Y)$, albeit when rotational inertia is present in the plate---\cite[p.690]{springer}. In this case, with $\alpha>0$, the necessary damping is of the form $k(1-\alpha \Delta)u_t$, but the only requirement is that $k>0$. Moreover, the result there is valid for $f(u)$ being the von Karman nonlinearity (as well as others with similar properties, including $f_B(u)$). We point out that in this case, the Neumann (plate to flow-mapping) is {\em compact}, and thus the transference of stability properties of the plate to the flow is more direct and natural. 

With respect to our model, namely with no inertia $\alpha=0$, we speculate that the approaches in \cite{conequil1,conequil2} are amenable here (for $f=f_B$) with $k>k_*$. Thus we believe the following holds:
		\begin{conjecture}
	Let the hypotheses of Theorem \ref{contoequil} be in force---notably, $k>k_*$. Then there exists a set $\mathbb L$ which is a finite determining set for the entire dynamics $(S_t,Y)$. \end{conjecture}

\vskip.1cm
\noindent{\bf Results for {\em Some} Damping}: While many results in the rotary inertia scenario hold for some damping, i.e., when the principal linear portion of the plate has the form
$$(1-\alpha\Delta)u_{tt}+\Delta^2w+k(1-\alpha\Delta)u_t+f(u)=p(\mathbf x,t),$$
we do not currently see a means of circumventing the need for large damping in the results and conjectures presented above. All results for imposed damping above represent an improvement of existing results through some form of decomposition of the plate dynamics into a smooth and exponentially decaying part---the latter portion unavoidably requires large damping. Physically, we would like results (e.g., subsonic convergence to equilibria or full system determining functionals) to hold for any amount of imposed damping, as there is some structural damping present in every elastic structure. Unfortunately, our methods do not seem to yield this at present. 

\vskip.1cm
\noindent{\bf Von Karman Nonlinearity}: As the Berger nonlinearity is a simplification---a physically accepted one for panels---of the scalar von Karman equations, one might naturally ask if the above results hold when $f(u)=f_V(u)$? The answer is, in general, complicated. We refer the reader to \cite{HLW} where these issues are discussed at length, and comparisons between von Karman and Berger dynamics in the absence of rotational inertia are the main theme. Two succinct comments are in order: 

(i) For von Karman's dynamics with $(k+1)>0$, showing the quasi-stability property is more difficult; in general, it can only be done on the attractor itself, rather than on the absorbing ball. If one assumes the damping is large, i.e., $k>k_*$, then one can obtain the quasi-stability estimate on the absorbing ball $\mathscr B$; this is in contrast to the situation here, where quasi-stability for the Berger system is shown for on the absorbing ball with only the damping coming from the flow (i.e., $k=0$), and its size does not matter. 

(ii) The second comment is that working in higher topologies for $f=f_V$, for instance in trying to construct smooth exponentially attracting sets, is far more difficult than the Berger dynamics. This leads to critical problems in most results presented in this section---see \cite{conequil2} for more discussion.

\vskip.1cm
\noindent{\bf Other Plate Boundary Conditions}: Unfortunately, while other boundary conditions are very interesting to us in the context of this model, the questions of modeling, well-posedness, and stability are very complicated. First, the modeling changes dramatically (i.e., von Karman/Berger) when free boundary conditions are utilized (see \cite{gw}). Secondly,  extension by zero is utilized in many arguments above, and this is not possible for hinged or free boundary conditions. We posit that this could be circumvented using other types of extensions, but the free boundary conditions are quite challenging and open for this type of coupled flow-plate model. In general, for non-homogeneous plate boundary conditions, there has been some work on the Berger flow-plate interaction with boundary damping through moments \cite{weblas}, but this question is also largely open. The survey type references \cite{fereisel,survey1,survey2} provide more detailed discussion, especially with regards to the free boundary condition.

\end{document}